\documentclass[a4paper,11pt]{amsart}

\usepackage{amssymb}
\usepackage{mathrsfs}
\usepackage{amsmath,amssymb,amsthm,latexsym,amscd,mathrsfs}
\usepackage{indentfirst}
\usepackage{stmaryrd}
\usepackage{graphicx}
\usepackage{extarrows}
\usepackage{multirow}
\usepackage{latexsym}
\usepackage{amsfonts}
\usepackage{color}
\usepackage{pictexwd,dcpic}
\usepackage{graphicx}
\usepackage{psfrag}
\usepackage{comment}
\usepackage{enumerate}
\usepackage{commath}
\usepackage[final]{showkeys}
\usepackage{hyperref}

\numberwithin{equation}{section}

\theoremstyle{theorem}
\newtheorem{thm}{Theorem}[section]
\newtheorem{prop}[thm]{Proposition}
\newtheorem{lem}[thm]{Lemma}
\newtheorem{cor}[thm]{Corollary}

\theoremstyle{definition}
\newtheorem{defn}{Definition}
\newtheorem{conj}{Conjecture}[section]
\newtheorem{rmk}{Remark}
\newtheorem{ack}{Acknowledgements}   
\makeatletter

\def\<{\langle}
\def\>{\rangle}
\def\({\left(}
\def\){\right)}
\def\[{\left[}
\def\]{\right]}

\DeclareMathOperator{\rank}{rank}

\DeclareMathOperator{\Span}{Span}

\makeatother

\newcommand{\x}[2]{{|x_{#1}-x_{#2}|}}
\newcommand{\y}[2]{{|y_{#1}-y_{#2}|}}
\newcommand{\yy}[2]{{|y'_{#1}-y'_{#2}|}}
\DeclareMathOperator{\sgn}{sgn}
\newcommand{\dist}[2]{{\mathop{\mathrm{dist}}(#1,#2)}}

\newcommand{\me}{\mathrm e}
\newcommand{\mi}{\mathrm i}
\def\tran{^{\mathrm T}}
\allowdisplaybreaks[1]

\title{On a Conjecture of Bahri-Xu}
\author[H. Chen]{Hong Chen}
\address{School of Mathematical Sciences, Beijing Normal University, Beijing 100875, P. R. China}
\email{hchen@mail.bnu.edu.cn}
\author[J.Q. Ge]{Jianquan Ge}
\address{School of Mathematical Sciences, Laboratory of Mathematics and Complex Systems, Beijing Normal University, Beijing 100875, P.R. CHINA.}
\email{jqge@bnu.edu.cn}
\author[K. Jia]{Kai Jia}
\address{School of Mathematical Sciences, Beijing Normal University, Beijing 100875, P. R. China}
\email{15935336026@163.com}
\author[Z.Q. Lu]{Zhiqin Lu}
\address{Department of Mathematics, University of California,
	Irvine, Irvine, CA 92697, USA.}
\email{zlu@math.uci.edu}
\date{}
\thanks{The second author is partially supported by Beijing Natural Science Foundation (Z190003). The last author is partially supported by National Sciences Foundation of USA (DMS-19-08513).}
\begin{document}
\maketitle
%-------------------------------------------------------------------------------
\begin{abstract}
	In order to study the Yamabe changing-sign problem, Bahri and Xu proposed a conjecture which is a universal inequality for $p$ points in $\mathbb R^m$. They have verified the conjecture for $p\leq3$.
	In this paper, we first simplify this conjecture by giving two sufficient and necessary conditions inductively.
	Then we prove the conjecture for the basic case $m=1$ with arbitrary $p$. In addition, for the cases when $p=4,5$ and $m\geq2$, we manage to reduce them to the basic case $m=1$ and thus prove them as well.
\end{abstract}
%------------------------------------------------------------------------------
%-------------------------------------------------------------------------------
%-------------------------------------------------------------------------------
\section{Introduction}
In the study of Yamabe problem on $S^3$, we consider a semi-linear equation
\begin{equation}\label{1}
	\Delta_{{\mathbb R}^3} u+u^5=0,\qquad u>0.
\end{equation}
It is well known that
\begin{equation*}
	\delta(a,\lambda)=\frac{c\sqrt\lambda}{(1+\lambda^2|x-a|^2)^{1/2}}
\end{equation*}
for any real number $\lambda$ and vector $a\in\mathbb R^3$ (with appropriate constant $c$) is a solution to the above equation. Moreover, for $\lambda_i$ large enough, the combinations
$\sum_{i=1}^p\delta(a_i,\lambda_i)$ are \emph{almost} solutions to the equation.

From PDE point of view, it would be interesting to study Equation~\eqref{1} without  the positivity assumption on $u$ as well. In the book ~\cite{BX07},  Bahri and Xu introduced the problem of studying the Morse Lemma at infinity. Morse Lemma is interesting in this situation because even for positive solutions, Equation~\eqref{1} is a  variational problem with defects. In order to study these defects, in~\cite{BX07}, the following question was introduced.

We consider the functional $J(u)$ defined by
\begin{equation*}
	J(u)=\left(\int_{\mathbb R^3} u^6\, \mathrm{d}x\right)^{-1}
\end{equation*}
on the space
\begin{equation*}
	\Sigma=\{w\mid \int_{\mathbb R^3}(|\nabla w|^2+w^6)dx<\infty,\int_{\mathbb R^3}|\nabla w|^2\, dx=1\}.
\end{equation*}

Let $\bar w_1,\cdots,\bar w_p$ be $p$ (possibly sign changing) solutions of ~\eqref{1}. Let $a_1,\cdots,a_p\in\mathbb R^3$ be $p$ points. Let $\lambda_1,\cdots,\lambda_p>0$ be very large numbers.
We wish to establish a formula for the functional $J$ that corresponds to the  following linear combination
\begin{equation*}
	\sum_{i=1}^p\alpha_i \sqrt{\lambda_i}\bar w_i(\lambda_i(x-a_i))
\end{equation*}
of solutions.
More precisely, we wish to establish an asymptotic expansion of the following
\begin{equation*}
	J\left(\sum_{i=1}^p\alpha_i\sqrt{\lambda_i}\bar w_i(\lambda_i(x-a_i))+v\right),
\end{equation*}
where $\alpha_i$ are constants and $v$ is a function satisfying Condition (Vo) in~\cite[Page 3]{BX07}.

In order to establish such an asymptotic expansion, Bahri and Xu made two additional assumptions (conjectures), one of which is purely linear algebraic, and can be stated as follows
\begin{conj}\label{conj}
	Let $x_1,\dots,x_p$ $(p\geq2)$ be distinct vectors in $\mathbb R^m$.
			Then there is a positive constant $c=c(p,m)$ depending only on $p$ and $m$, such that
	\begin{align}\label{eqn:conj}
		|AU|^2 + \sup_{1\leq i\leq p} \left|U^T\(\frac{\partial A}{\partial x_i}\)U\right|
		\geq c(p,m) \sum_{\substack{i,j\\j\neq i}} \frac{u_j^2}{|x_i-x_j|^2}
	\end{align}
	for any $U=\(u_1,\dots,u_p\)\tran\in\mathbb R^p$, where $A=(a_{ij})_{1\leq i,j\leq p}$ is a $p\times p$ matrix with entries
	\begin{align*}
		a_{ij}=\begin{cases}
		0,					&i=j	\\
		\frac{1}{\x{i}{j}}, &i\neq j
		\end{cases}
	\end{align*}
	that is,
	\begin{align*}%\label{eqn:conj-A}
		A=\begin{pmatrix}
		0 & \frac{1}{|x_1-x_2|} & \cdots & \frac1{|x_1-x_p|} \\
		\frac1{\x21} & 0 & \cdots & \frac1{\x2p}\\
		\vdots & \vdots & \ddots & \vdots \\
		\frac1{|x_p-x_1|} & \frac1{|x_p-x_2|} & \cdots & 0
		\end{pmatrix},
	\end{align*}
	and
	\begin{align*}
		\dpd{A}{x_k}
		=	\(\dpd{a_{ij}}{x_k}\)_{1\leq i,j\leq p}	
	\end{align*}
	is a vector-valued $p\times p$ matrix.
\end{conj}

Equivalently, \eqref{eqn:conj} can be rewritten as
\begin{align}\label{eqn:rewrite}
	\sum_{1\leq i\leq p} \left| \sum_{j\neq i} \frac{u_j}{\x ij} \right|^2
	+ 2\sup_{1\leq i\leq p} \left|\sum_{j\neq i} u_i \frac{x_i-x_j}{\x ij^3} u_j \right|
	\geq c(p,m) \sum_{\substack{i,j\\j\neq i}} \frac{u_j^2}{\x ij^2}.
\end{align}
Let $I_1,I_2$ denote the two hands of (\ref{eqn:rewrite}):
\begin{align*}
	I_1 &= \sum_{1\leq i \leq p} \left|\sum_{j\neq i} \frac{u_j}{|x_i-x_j|}\right|^2
	+2\sup_{1\leq i\leq p}\left|\sum_{j\neq i} u_i \frac{x_i-x_j}{\x ij^3} u_j\right|,\\	%\label{eqn:l1-I1}
	I_2 &= \sum_{\substack{i,j\\j\neq i}} \frac{u_j^2}{\x ij^2}. %\label{eqn:l1-I2}
\end{align*}

\begin{rmk}
	The conjecture was stated in~\cite[page 4, Conjecture 2]{BX07}, where $m=3$. We extend the conjecture from 3 dimensional to arbitrary dimensional. We found that even in the case $m=1$, that is, when all $x_i$ are real numbers instead of vectors, the conjecture is still interesting and open.
\end{rmk}
\begin{rmk}\label{rmk:m}
	Due to the symmetries of (\ref{eqn:rewrite}), we only need to consider the conjecture for $m\leq p-1$.
	In fact, if $m\geq p$, we may assume that $x_1,\dots,x_p$ lie in $\mathbb R^{p-1}\hookrightarrow\mathbb R^m$ since any $p$ points in $\mathbb R^m$ must lie in some affine subspace of dimension $p-1$ and  (\ref{eqn:rewrite}) is invariant under ambient similarities of $\mathbb R^m$.
\end{rmk}

The main purpose of this paper is to study Conjecture \ref{conj}.
If  $p = 2$, then $|AU|^2=\frac{u_1^2+u_2^2}{\x12^2}=I_2$, and thus the conjecture is valid for $c(p,m)=1$.

In~\cite{Xu05}, Xu proved the case $p=m=3$. By the above remark, Xu's result actually implies the cases of $p=3$ with arbitrary $m$.

From now on, we are only concerned with $p\geq4$.

In this paper, we shall prove some equivalent conditions for Conjecture~\ref{conj}, see Theorems~\ref{thm:A} and \ref{thm:B}.
In addition, we shall prove the conjecture in some special cases, namely, the basic case $m=1$ with arbitrary $p$; and $p=4,5$ with arbitrary $m$.

In Section~\ref{sec:2} we give the following equivalent characterization for Conjecture~\ref{conj}.
\begin{thm}\label{thm:A}
	Let $p_0\geq4$.
	Conjecture~\ref{conj} holds for any $4\leq p\leq p_0$,
	if and only if for any $4\leq p\leq p_0$ and distinct $x_1,\dots,x_p\in\mathbb R^m$, the equations
	\begin{align}
		\sum_{j\neq i} \frac{u_j}{\x{i}{j}} &= 0,	\quad 1\leq i\leq p	\label{eqn:equiv1},\\\intertext{and,}
		u_i\sum_{j\neq i} \frac{x_i-x_j}{\x{i}{j}^3} u_j &=0,	\quad 1\leq i\leq p	\label{eqn:equiv2},
	\end{align}
	about $u_1,\dots,u_p$ have NO non-zero solution.
\end{thm}
%The proof is based on induction on $p_0$: $p_0=4$ is the induction base; Theorem~\ref{thm:equiv} is the induction step from $p_0-1$ to $p_0$.
%To prove theorem~\ref{thm:equiv}, we need Lemmas~\ref{lemma1} and \ref{lemma2}.

In Section~\ref{sec:3} we firstly prove the following simplified equivalent conditions, using which we then prove the special cases mentioned before.
\begin{thm}\label{thm:B}
	Let $p_0\geq4$.
	Conjecture~\ref{conj} holds for $4\leq p\leq p_0$, if and only if for $4\leq p\leq p_0$, given any distinct $y_1,\dots,y_p\in \mathbb S^m$, the equations
	\begin{align}
		\sum_{j\neq i} \frac{y_i-y_j}{\y ij^3} v_j =0,	\quad 1\leq i\leq p	\label{eqn:equiv2''}
	\end{align} for $v_1,\dots,v_p$ have NO non-zero solution.
\end{thm}

%-------------------------------------------------------------------------------
%-------------------------------------------------------------------------------
%-------------------------------------------------------------------------------

\section{Proof of Theorem~\ref{thm:A}}\label{sec:2}
In this section, we prove Theorem~\ref{thm:A} by induction on $p_0$.
Assuming that Conjecture~\ref{conj} holds for $p\leq p_0-1$ ($p_0\geq4$), we shall consider the equivalence when $p=p_0$. The difficulty lies in the analysis when there are singularities, namely, one or more pairs of vectors with distance $|x_i-x_j|$ going to zero or infinity. This would be also an obstacle if one tried to prove the equivalence by contradiction. Hence from the inductive viewpoint, we treat with these two cases of singularities in Lemmas \ref{lemma2} and \ref{lemma1} separately.

First, observe that the inequality is homogeneous with respect to $x_1,\dots,x_p$.
Thus, without loss of generality, we may assume that $\x12 = \min\limits_{i\neq j}\x ij=1$, and
$$
	1=\x12\leq\x13\leq\dots\leq\x1p.
$$

%-------------------------------------------------------------------------------

The following Lemma~\ref{lemma1} deals with \textit{unbounded} cases.
\begin{lem}\label{lemma1}%----------------------------------------------------
%	Assume that Conjecture~\ref{conj} holds for $p\leq p_0-1$. When $p=p_0\geq4$, let $2\leq s\leq p-2$.
%	If there exists a sequence of positive numbers $\beta_1,\dots,\beta_{s-1}$ depending only on $p$ and $m$, such that the following condition holds:
	Assume that Conjecture~\ref{conj} holds for $p\leq p_0-1$.
	When $p=p_0\geq4$, let $2\leq s\leq p-2$.
	For any $\beta_1,\dots,\beta_{s-1}\geq1$, there exists $ M_s>0$, such that for any $x_1,\dots,x_s$ satisfying
	\begin{align}
		\x{1}{i+1} \leq \beta_i,	\quad	1\leq i \leq s-1,		\tag{$\star$}\label{eqn:condition}
	\end{align}
	Conjecture~\ref{conj} holds if $\x{1}{s+1}>M_s$.
	
	\begin{comment}
		??设猜想~\ref{conj}在$p\leq p_0-1$时?立. 设$p=p_0\geq4$, $2\leq s\leq p-2$. 如果存在一个???赖于$p,m$的正数列$\beta_1,\dots,\beta_{s-1}$, 满足$\x{1}{i+1} \leq \beta_i$, $i = 1,2,\dots, s-1$, 那么存在$M_s>0$, 使得猜想~\ref{conj} 在$\x{1}{s+1}>M_s$ 时?立.
	\end{comment}
\end{lem}

\begin{proof}
%	Let $\beta_1,\dots,\beta_{s-1}$ be such a sequence, then $\x{1}{i+1} \leq \beta_i$, for $i = 1,2,\dots, s-1$.
	Fix any $\beta_1,\dots,\beta_{s-1}\geq1$, then $\x{1}{i+1} \leq \beta_i$, for $i = 1,2,\dots, s-1$.
	Let $\alpha = \beta_{s-1}$, then $1\leq\x{1}{s}\leq \alpha$.
	In this proof, we consider only the case $\x1{s+1}\geq2\alpha$.

%-------------------------------------------------------------------------------
	
%\paragraph{一些?等?}
%	用上??的??件和三角?等?, 立????以得到一些简??的?等?, 它们在????的?明中会??挥??作用.
	Firstly, we give some important inequalities that would be useful in the following proof.
	For $j\leq s,i\geq s+1$,
	\begin{align}\label{eqn:l1-ineq1}
				|x_i-x_j|
		&\geq	|x_1-x_i| - |x_1-x_j| \notag\\	
		&\geq	2\alpha - |x_1-x_j| \notag\\
		&\geq	2|x_1-x_j|-|x_1-x_j|	\notag\\
		&=	|x_1-x_j|;
	\end{align}	
	\begin{align}\label{eqn:l1-ineq2}
				|x_i-x_{s+1}|
		&\leq	|x_i-x_1| + |x_1-x_{s+1}| \nonumber\\
		&\leq	2|x_i-x_1| \nonumber\\
		&\leq	2|x_i-x_j| + 2|x_j-x_1| \nonumber\\
		&\leq	4|x_i-x_j|;
	\end{align}
	\begin{align}\label{eqn:l1-ineq3}
		|x_j-x_{s+1}| \leq 5|x_i-x_j|.
	\end{align}

%-------------------------------------------------------------------------------
	
	Now, we deal with $I_1$.
	%\paragraph{把$I_1$放缩至?$s+1$项}
	%\subparagraph{$I_1$第一项}
	%?明的第一步是将$I_1$放缩至?$s+1$项.
	Since $\left(a+2b\right)^2 = a^2+4ab+4b^2 = 2[(a+b)^2-(a^2/2-b^2)]\geq0$, we have
	\begin{align}\label{eqn:l1-I1-1}
				\sum_{1\leq i\leq p} \left| \sum_{j\neq i} \frac{u_j}{|x_i-x_j|} \right|^2
		&\geq	\sum_{i\leq s+1} \left| \sum_{j\neq i} \frac{u_j}{|x_i-x_j|} \right|^2 \nonumber\\
		&\geq	\frac12 \sum_{i\leq s+1} \left| \sum_{j\neq i,j\leq s+1} \frac{u_j}{|x_i-x_j|} \right|^2
			-\sum_{i\leq s+1} \left| \sum_{j>s+1} \frac{u_j}{|x_i-x_j|} \right|^2.
	\end{align}
	%\subparagraph{$I_1$第二项}
	Since
	$$
		\sup_{i\leq s+1} |a_i+b_i|
		\geq |a_k+b_k|
		\geq |a_k|-|b_k|
		\geq |a_k| - \sup_{i\leq s+1}|b_i|,\quad \forall k\leq s+1,
	$$
	we have
	$$
		\sup_{i\leq s+1} |a_i+b_i| \geq \sup_{i\leq s+1}|a_i| -\sup_{i\leq s+1}|b_i|.
	$$
	Thus,
	\begin{align}\label{eqn:l1-sup-1}
		&		2\sup_{1\leq i\leq p} \left|\sum_{j\neq i} u_i \frac{x_i-x_j}{\x ij^3} u_j \right| \notag\\
		&\geq	2\sup_{1\leq i\leq s+1} \left|\sum_{j\neq i} u_i \frac{x_i-x_j}{\x ij^3} u_j \right| \notag\\
		&\geq	2\sup_{1\leq i\leq s+1} \left|\sum_{j\neq i, j\leq s+1} u_i \frac{x_i-x_j}{\x ij^3} u_j \right|
				-2\sup_{1\leq i\leq s+1} \left|\sum_{j\neq i, j>s+1} u_i \frac{x_i-x_j}{\x ij^3} u_j \right|.
	\end{align}
	%\subparagraph{$I_1$放缩$I_1'$}
	By \eqref{eqn:l1-I1-1} and \eqref{eqn:l1-sup-1}, we have
	\begin{align*}%\label{eqn:l1-I1-2}
		I_1 \geq \frac12 \sum_{i\leq s+1} \left| \sum_{j\neq i,j\leq s+1} \frac{u_j}{|x_i-x_j|} \right|^2
		+2\sup_{1\leq i\leq s+1} \left|\sum_{j\neq i, j\leq s+1} u_i \frac{x_i-x_j}{\x ij^3} u_j \right| \nonumber\\
		- \sum_{i\leq s+1} \left| \sum_{j>s+1} \frac{u_j}{|x_i-x_j|} \right|^2
		-2\sup_{1\leq i\leq s+1} \left|\sum_{j\neq i, j>s+1} u_i \frac{x_i-x_j}{\x ij^3} u_j \right|.
	\end{align*}
%-------------------------------------------------------------------------------

	Since $s+1\leq p-1$, by the inductive assumption, there exists constant $c_1=c_1(p,m)>0$, such that 	
%	?用第二数学归纳法，??设猜想~\ref{conj} 对$p'\leq p-1$ ?立， 由于$s+1<p$, 故存在常数$c_1$,使得
	\begin{align}\label{eqn:l1-I1-3}
		&I_1 \geq c_1 \sum_{i,j\leq s+1,j\neq i}\frac{u_j^2}{|x_i-x_j|^2} \nonumber\\
		& - \sum_{i\leq s+1} \left| \sum_{j>s+1} \frac{u_j}{|x_i-x_j|} \right|^2
		-2\sup_{1\leq i\leq s+1} \left|\sum_{j\neq i, j>s+1} u_i \frac{x_i-x_j}{\x ij^3} u_j \right|.%=:I_1'
	\end{align}
	
	Consider the middle term in \eqref{eqn:l1-I1-3}, by the special case of the power means inequality, we have
	$$
				\frac{1}{p-s-1}\sum_{j>s+1} \frac{u_j}{|x_i-x_j|}
		\leq	\sqrt{	\frac{1}{p-s-1}	\sum_{j>s+1}\frac{u^2_j}{|x_i-x_j|^2}	},
	$$
	thus,
	\begin{align}\label{eqn:l1-10}
				\sum_{i\leq s+1} \left| \sum_{j>s+1} \frac{u_j}{|x_i-x_j|} \right|^2
		\leq 	(p-s-1)\sum_{j>s+1,i\leq s+1}\frac{u^2_j}{|x_i-x_j|^2}.
	\end{align}
	
	%\subparagraph{考察$I_1'$的第二个?项}
	Consider the last term in \eqref{eqn:l1-I1-3}.
	When $i=s+1$, then
	\begin{align}\label{eqn:l1-sup-3}
				2\left|\sum_{j>s+1} u_{s+1} \frac{x_{s+1}-x_j}{|x_{s+1}-x_j|^3} u_j \right|
		&\leq	\sum_{j>s+1} \frac{2|u_{s+1}u_j|}{|x_{s+1}-x_j|^2} \nonumber\\
		&\leq	\sum_{j>s+1} \frac{u_{s+1}^2}{|x_{s+1}-x_j|^2} + \sum_{j>s+1} \frac{u_j^2}{|x_{s+1}-x_j|^2}.
	\end{align}	
	When $i\leq s$, since $\x ij\geq1$,
	\begin{align}\label{eqn:l1-sup-4} %\label{eqn:l1-sup-2}
				2\left|\sum_{j>s+1} u_i \frac{x_i-x_j}{|x_i-x_j|^3} u_j \right|	
		\leq	\sum_{j>s+1}\frac{ 2|u_iu_j| }{|x_i-x_j|^2}
		\leq	\sum_{j>s+1} \frac{2|u_i||u_j|}{|x_i-x_j|}.
	\end{align}
	%
	%由平??值?等?, 有
	By the inequality of arithmetic and geometric means,
	\begin{align*}
		\frac 1a |u_i|^2 + a\frac{|u_j|^2}{|x_i-x_j|^2} \geq \frac{2|u_i||u_j|}{|x_i-x_j|},
	\end{align*}
	taking $a=(p-s-1)\dfrac{2\alpha^2}{c_1}>0$, we get
	\begin{align}\label{eqn:l1-sup-5}
				\sum_{j>s+1} \frac{2|u_i||u_j|}{|x_i-x_j|}
		\leq	\sum_{j>s+1} \frac 1a u_i^2 + \sum_{j>s+1}a\frac{u_j^2}{|x_i-x_j|^2} %\nonumber\\
		\leq	\frac{c_1}{2\alpha^2} u_i^2 + a\sum_{j>s+1,i\leq s} \frac{u_j^2}{|x_i-x_j|^2}.
	\end{align}
	Now take $i_0: 1\leq i_0\leq s$, such that $u_{i_0}^2=\max_{1\leq i\leq s} u^2_i$.
	If $i_0=1$, we take the term $i=2,j=1$ in the sum $\frac12 c_1\sum_{i,j\leq s+1,j\neq i}\frac{u_j^2}{|x_i-x_j|^2}$, then
	\begin{align*}
		\frac12 c_1\sum_{i,j\leq s+1,j\neq i}\frac{u_j^2}{|x_i-x_j|^2}
		\geq\frac{c_1u_1^2}{2|x_2-x_1|^2}
		\geq\frac{c_1}{2\alpha^2} u_{i_0}^2.
	\end{align*}
	If $2\leq i_0\leq s$, we take the term $i=1,j=i_0$, then
	\begin{align*}
		\frac12 c_1\sum_{i,j\leq s+1,j\neq i}\frac{u_j^2}{|x_i-x_j|^2}
		\geq\frac{c_1u_{i_0}^2}{2|x_{i_0}-x_1|^2}
		\geq\frac{c_1}{2\alpha^2} u_{i_0}^2.
	\end{align*}
	In either case, we have
	\begin{align}\label{eqn:l1-1}
		\frac{c_1}{2\alpha^2} u_{i_0}^2
		\leq	\frac12 c_1\sum_{i,j\leq s+1,j\neq i}\frac{u_j^2}{|x_i-x_j|^2}
	\end{align}
	By \eqref{eqn:l1-sup-3},\eqref{eqn:l1-sup-4},\eqref{eqn:l1-sup-5} and \eqref{eqn:l1-1}, the last term in \eqref{eqn:l1-I1-3} has the following upper bound:
	\begin{align}\label{eqn:l1-sup-6}
		&2\sup_{1\leq i\leq s+1} \left|\sum_{j\neq i, j>s+1} u_i \frac{x_i-x_j}{\x ij^3} u_j \right| \notag\\
		&\leq \frac12 c_1\sum_{i,j\leq s+1,j\neq i}\frac{u_j^2}{|x_i-x_j|^2}
		+ c_2\sum_{j>s+1,i\leq s+1}\frac{u^2_j}{|x_i-x_j|^2}
		+ \sum_{j>s+1}\frac{u_{s+1}^2}{|x_{s+1}-x_j|^2},
	\end{align}
	where $c_2=\max\{1,a\}>0$.
	
%-------------------------------------------------------------------------------

	%\subparagraph{代入$I_1'$}
	By $\eqref{eqn:l1-I1-3}$, \eqref{eqn:l1-10} and $\eqref{eqn:l1-sup-6}$, we have
	\begin{align}\label{eqn:l1-I1-4}
		I_1
		&\geq	\frac{c_1}2 \sum_{i,j\leq s+1,j\neq i}\frac{u_j^2}{|x_i-x_j|^2} \notag\\
		&		-\left( (p-s-1+c_2)\sum_{j>s+1,i\leq s+1} \frac{u^2_j}{|x_i-x_j|^2} + \sum_{j>s+1}\frac{|u_{s+1}|^2}{|x_{s+1}-x_j|^2} \right)	\nonumber\\
		&\geq	\frac{c_1}{2} \sum_{i,j\leq s+1,j\neq i} \frac{u_j^2}{|x_i-x_j|^2}
		-c_3 \left( \sum_{j>s+1,i\leq s+1} \frac{u_j^2}{|x_i-x_j|^2} + \sum_{j>s+1} \frac{u_{s+1}^2}{|x_{s+1}-x_j|^2} \right),
	\end{align}
		where $c_3=\max\{p-s-1+c_2,1\}=p-s-1+c_2\geq2$.

%-------------------------------------------------------------------------------

	Now, consider the minus term in $\eqref{eqn:l1-I1-4}$.
	Rewrite and then use inequalities \eqref{eqn:l1-ineq1}, \eqref{eqn:l1-ineq2} and \eqref{eqn:l1-ineq3}:
	\begin{align*}
		&		\sum_{j>s+1,i\leq s+1} \frac{u_j^2}{|x_i-x_j|^2}
				+\sum_{j>s+1} \frac{u_{s+1}^2}{|x_{s+1}-x_j|^2} \\
		&=		\sum_{j>s+1,i\leq s} \frac{u_j^2}{|x_i-x_j|^2}
				+\sum_{j>s+1,i=s+1} \frac{u_j^2}{|x_i-x_j|^2}
				+\sum_{i>s+1,j=s+1} \frac{u_j^2}{|x_i-x_j|^2} \\
		&\leq	16\sum_{j>s+1,i\leq s} \frac{u_j^2}{|x_{s+1}-x_j|^2}
				+\sum_{j>s+1,i=s+1} \frac{u_j^2}{|x_i-x_j|^2}
				+\sum_{i>s+1,j=s+1} \frac{u_j^2}{|x_i-x_j|^2} \\
		&=		(16s+1) \sum_{j>s+1,i=s+1} \frac{u_j^2}{|x_i-x_j|^2}
				+ \sum_{i>s+1,j=s+1} \frac{u_j^2}{|x_i-x_j|^2} \\
		&\leq	(16s+1) \sum_{i,j\geq s+1,i\neq j} \frac{u_j^2}{|x_{i}-x_j|^2}.
	\end{align*}
	%16s+2????
	Set $c_4=(16s+1)c_3>0$, then $\eqref{eqn:l1-I1-4}$ becomes
	\begin{align}\label{eqn:l1-I1-7}
	I_1 \geq \frac{c_1}{2} \sum_{i,j\leq s+1,j\neq i} \frac{u_j^2}{|x_i-x_j|^2} -c_4 \sum_{i,j\geq s+1,j\neq i} \frac{u_j^2}{|x_i-x_j|^2}.
	\end{align}
	
%-------------------------------------------------------------------------------

	%\paragraph{把$I_1$放缩至??$p−s−1$项}	
	%至此, ?明的第一步就完?了, 第二步是将$I_1$放缩至??$p-s-1$项.
	
	By similar argument as in \eqref{eqn:l1-I1-1} and \eqref{eqn:l1-sup-1}, we have
	\begin{align*}%\label{eqn:l1-I1-5}
		I_1 \geq \frac12 \sum_{i>s} \left| \sum_{j\neq i,j>s} \frac{u_j}{|x_i-x_j|} \right|^2
		+2\sup_{i>s} \left|\sum_{j>s, j\neq i} u_i \frac{x_i-x_j}{\x ij^3} u_j \right| \nonumber\\
		- \sum_{i>s} \left| \sum_{j\leq s} \frac{u_j}{|x_i-x_j|} \right|^2
		-2\sup_{i>s} \left|\sum_{j\leq s} u_i \frac{x_i-x_j}{\x ij^3} u_j \right|.
	\end{align*}
	By the inductive assumption, there exists $c_5>0$ such that
	\begin{align}\label{eqn:l1-I1-9}
		I_1 \geq c_5 \sum_{i,j>s,j\neq i} \frac{u_j^2}{|x_i-x_j|^2} %\nonumber\\
		- \sum_{i>s} \left| \sum_{j\leq s} \frac{u_j}{|x_i-x_j|} \right|^2
		-2\sup_{i>s} \left|\sum_{j\leq s} u_i \frac{x_i-x_j}{\x ij^3} u_j \right|.
	\end{align}

	We claim that given any $\varepsilon>0$, there exists $b_\varepsilon\geq1$, such that
	\begin{align}\label{eqn:l1-I2-1}
		\sum_{i>s} \left| \sum_{j\leq s} \frac{u_j}{|x_i-x_j|} \right|^2 + 2\sup_{i>s} \left|\sum_{j\leq s} u_i \frac{x_i-x_j}{\x ij^3} u_j \right|
		\leq \varepsilon I_2
	\end{align}
	when $\x ij>b_\varepsilon$ for $i>s,j\leq s$.
	
	In fact, by homogeneity we can assume $|u_i|\leq1$ for all $i$ and thus (\ref{eqn:l1-I2-1}) can be estimated as follows: the left hand is
\begin{align*}
& \sum_{i>s} \left| \sum_{j\leq s} \frac{u_j}{|x_i-x_j|} \right|^2 + 2\sup_{i>s} \left|\sum_{j\leq s} u_i \frac{x_i-x_j}{\x ij^3} u_j \right| \nonumber\\
&\leq \sum_{i>s} \Big(s \sum_{j\leq s} \frac{u_j^2}{|x_i-x_j|^2} +2s\sum_{j\leq s} \frac{u_j^2}{|x_i-x_j|^2}\Big) \nonumber\\
&\leq \frac{3s(p-s+1)}{b_{\varepsilon}^2} \sum_{j\leq s}u_j^2,
\end{align*}
while the right hand is
\begin{align*}
I_2\geq \sum_{j\leq s} \Big(\sum_{i\neq j}\frac{1}{|x_i-x_j|^2}\Big)u_j^2\geq \frac{1}{4\alpha^2}\sum_{j\leq s}u_j^2.
\end{align*}
Therefore, setting $$b_{\varepsilon}=\sqrt{\frac{3s(p-s+1)4\alpha^2}{\varepsilon}},$$
we get the required inequality (\ref{eqn:l1-I2-1}).

%if $u_j=0$, $\forall j:1\leq j\leq s$, then $LHS=0\leq I_2$;
%	otherwise $\exists j:1\leq j\leq s$, $u_j\neq 0$, then $I_2 \geq \frac{u_j^2}{\alpha^2}>0$.
%	Given any $b>0$, consider the case $\x ij>b$ for $i>s,j\leq s$.
%	Then $\frac1{\x ij}\to0$, as $b\to+\infty$, thus
%	$$
%		\sum_{i>s} \left|\sum_{j\leq s}\frac{u_j}{|x_i-x_j|}\right|^2	+2\sup_{i>s} \left|	\sum_{j\leq s} u_i \frac{x_i-x_j}{\x ij^3}u_j	\right|	\to0,	\quad b\to+\infty.
%	$$
%	%The claim follows from the definition of limit.
%	
%	Now, let $M_s(\varepsilon) = (b_\varepsilon+1)\alpha \geq 2\alpha$.
%	In the rest of the proof, we consider only the case $\x1{s+1}>M_s(\varepsilon)$.
%	Check that
%	\begin{align*}%\label{eqn:l1-si1}
%				|x_i-x_j|	
%		&\geq	|x_1-x_i|-|x_1-x_j| \notag\\
%		&\geq	|x_1-x_{s+1}|-|x_1-x_s| \notag\\	
%		&\geq	\(\frac{M_s(\varepsilon)}{\alpha}-1\) |x_1-x_s| \notag\\
%		&\geq	\frac{M_s(\varepsilon)}{\alpha}-1=b_\varepsilon,
%		%&\geq	(\beta_s-1) |x_1-x_s| \notag\\
%		%&\geq	\beta_s-1.
%		\quad 	i>s,j\leq s.
%	\end{align*}
%	Thus, the inequality \eqref{eqn:l1-I2-1} holds.
	
	By \eqref{eqn:l1-I1-9} and \eqref{eqn:l1-I2-1},
	\begin{align}\label{eqn:l1-I1-6}
		I_1	\geq c_5 \sum_{i,j>s,j\neq i} \frac{u_j^2}{|x_i-x_j|^2} -\varepsilon I_2.
	\end{align}

%-------------------------------------------------------------------------------

	%\paragraph{$I_1$的???加起??}
	%?明的第三步
	Now, let $\delta=\frac{c_5}{2c_4}$, then $\delta\cdot\eqref{eqn:l1-I1-7}+\eqref{eqn:l1-I1-6}:$
	\begin{align}\label{eqn:l1-I1-8}
				(1+\delta)I_1
		&\geq	\frac{c_1\delta}2 \sum_{i,j\leq s+1,j\neq i} \frac{u_j^2}{|x_i-x_j|^2} + \frac{c_5}2 \sum_{i,j>s,j\neq i} \frac{u_j^2}{|x_i-x_j|^2} -\varepsilon I_2 \nonumber\\
		&\geq	c_6\( \sum_{i,j\leq s+1,j\neq i} \frac{u_j^2}{|x_i-x_j|^2} + \sum_{i,j>s,j\neq i} \frac{u_j^2}{|x_i-x_j|^2}\)  -\varepsilon I_2 .
	\end{align}
	where $c_6=\min\left\{\frac{c_1\delta}{2},\frac{c_5}2\right\}>0.$

	By \eqref{eqn:l1-ineq2} and \eqref{eqn:l1-ineq3},
	\begin{align*}
				\sum_{j>s+1,i\leq s} \frac{u_j^2}{|x_i-x_j|^2}
		\leq	16s \sum_{j>s+1} \frac{u_j^2}{|x_j-x_{s+1}|^2}
		\leq	16s \sum_{i,j>s,i\neq j} \frac{u_j^2}{|x_i-x_j|^2},
	\end{align*}
	\begin{align*}
				\sum_{j\leq s,i>s+1} \frac{u_j^2}{|x_i-x_j|^2} 	
		&\leq	25 \sum_{j\leq s,i>s+1} \frac{u_j^2}{|x_j-x_{s+1}|^2} \\
		&=		25(p-s-1) \sum_{j\leq s} \frac{u_j^2}{|x_j-x_{s+1}|^2}	\\
		&\leq	25(p-s-1) \sum_{i,j\leq s+1,i\neq j} \frac{u_j^2}{|x_i-x_j|^2}.
	\end{align*}
	Thus,
	\begin{align}\label{eqn:l1-I2-2}
				I_2
		&\leq	(16s+1) \sum_{i,j>s} \frac{u_j^2}{|x_i-x_j|^2} + \(25(p-s-1)+1\)\sum_{i,j\leq s+1,i\neq j} \frac{u_j^2}{|x_i-x_j|^2}\nonumber\\
		&\leq	c_7 \(\sum_{i,j>s} \frac{u_j^2}{|x_i-x_j|^2} + \sum_{i,j\leq s+1,i\neq j} \frac{u_j^2}{|x_i-x_j|^2}\),
	\end{align}
	where $c_7=\max\left\{16s+1,25(p-s-1)+1\right\}>0$.

	For $\varepsilon=\dfrac{c_6}{2c_7}$, choose $M_s=M_s(\varepsilon)\geq2\alpha$.
	By \eqref{eqn:l1-I1-8} and \eqref{eqn:l1-I2-2}, we have
	\begin{align*}%\label{eqn:l1-I1I2}
		I_1	\geq \frac{c_6}{2(1+\delta)c_7}I_2=cI_2,
	\end{align*}
	where $c=\dfrac{c_6}{2(1+\delta)c_7}>0$ depends only on $p$ and $m$.
\end{proof}

Now, let $s=2$ and take $\beta_1=1$.
Then by Lemma~\ref{lemma1}, there exists $M_2>0$ such that Conjecture~\ref{conj} holds if \eqref{eqn:condition} holds (automatically) and $\x13>M_2$.
Then it suffices to consider the case $\x13\leq M_2$.

Similarly, let $s=3$ and take $\beta_2=M_2$. Then by Lemma~\ref{lemma1}, there exists $M_3>0$ such that Conjecture~\ref{conj} holds if \eqref{eqn:condition} holds and $\x14>M_3$.
Then it suffices to consider the case $\x14\leq M_3$.

Continuing the process repeatedly, we reach $s=p-2$ and have a sequence $\beta_1,\dots,\beta_{p-3}$.
Again, by Lemma~\ref{lemma1}, there exists $M_{p-2}>0$ such that Conjecture~\ref{conj} holds if \eqref{eqn:condition} holds and $\x1{p-1}>M_{p-2}$.
It suffices to consider the case $\x1{p-1}\leq M_{p-2}$.

Therefore, from now on, we may assume that
$$
	1\leq|x_i-x_j|\leq\beta,	\quad	1\leq i,j\leq p-1.
$$

%------------------------------------------------------------------------------

Given $x_1,\dots,x_{p-1},$ let $c_1(p,m,x_1,\dots,x_{p-1})$ denote the infimum of
\begin{align*}
	&\left(
		\sum_{\substack{i,j\leq p-1\\i\neq j}} \frac{u_j^2}{|x_i-x_j|^2} + \tilde{u}_p^2	\right)^{-1}
	\\
	&\times\left(
		\sum_{i\leq p-1} \left|\sum_{j\leq p-1,j\neq i,} \frac{u_j}{|x_i-x_j|}+\tilde{u}_p\right|^2	+2\sup_{i\leq p-1}  \left|u_i \sum_{j\leq p-1,j\neq i} \frac{x_i-x_j}{|x_i-x_j|^3}u_j\right|
	\right),
\end{align*}
where $u^2_1+\dots+u^2_{p-1}+\tilde{u}_p^2=1$.
%Let $c(p,m,x_1,\dots,x_p)$ denote the infimum of
%\begin{align*}
%	\left( \sum_{\substack{i,j\\i\neq j}} \frac{u_j^2}{|x_i-x_j|^2}\right)^{-1}%
%	\times\left(k
%	\sum_{i\leq p} \left|\sum_{j\leq p,j\neq i,} \frac{u_j}{|x_i-x_j|}\right|^2
%	+ 2\sup_{i\leq p}  \left|u_i \sum_{j\leq p,j\neq i} \frac{x_i-x_j}{|x_i-x_j|^3}u_j\right|
%	\right)
%\end{align*}
%where $u^2_1+\dots+u^2_{p}=1$.

\begin{lem}\label{lemma2}
	Assume Conjecture~\ref{conj} holds for $p\leq p_0-1$.
	If $c_1(p,m,x_1,\dots,x_{p-1})>0$ holds for any $x_1,\dots,x_{p-1}$, then Conjecture~\ref{conj} is valid for $p=p_0$.
\end{lem}

\begin{proof}	
	%因为$c_1(p,m,x_1,\dots,x_{p-1})$是关于$x_1,\dots,x_{p-1}$的连续函数, 于是存在???赖$p,m$的正数$c_1(p,m)$使得
	Since $c_1(p,m,x_1,\dots,x_{p-1})$ is continuous relative to $x_1,\dots,x_{p-1}$, there exists positive number $c_1(p,m)$ depending only on $p,m$ such that
	\begin{align*}
		c_1(p,m,x_1,\dots,x_{p-1}) \geq 2c_1(p,m)>0.
	\end{align*}
	Thus,
	\begin{align}\label{eqn:l2-I1-1}
		&		\sum_{i\leq p-1} \left| \sum_{j\leq p-1,j\neq i} \frac{u_j}{\x ij} +\tilde{u_p}	\right|^2 + 2\sup_{i\leq p-1} \left| u_i\sum_{j\leq p-1,j\neq i}\frac{x_i-x_j}{\x ij ^3}u_j	\right| \notag\\
		&\geq	2c_1(p,m) \( \sum_{i,j\leq p-1,j\neq i} \frac{u_j^2}{\x ij^2}+\tilde{u_p}^2	\),
	\end{align}
	for $ u_1^2+\dots+u_{p-1}^2+\tilde{u_p}^2=1$.
	Let $u_p=\tilde{u_p}\x 1p$.
	
%-------------------------------------------------------------------------------

	Similar to the argument as in the proof of Lemma~\ref{lemma1}, we have
	\begin{align}\label{eqn:l2-I1-2}
				\sum_{i\leq p} \left| \sum_{j\leq p,j\neq i}\frac{u_j}{\x ij}	\right|^2
		&\geq	\sum_{i\leq p-1} \left|	\(\sum_{j\leq p-1,j\neq i} \frac{u_j}{\x ij} +\tilde{u_p}\)	+ \(\frac{u_p}{\x ip} -\tilde{u_p}\) 	\right|^2 \notag\\
		&\geq	\frac12\sum_{i\leq p-1} \left| \sum_{j\leq p-1,j\neq i}\frac{u_j}{\x ij} +\tilde{u_p}	\right|^2
		-2\sum_{j\leq p-1} \left| \frac{\x 1p}{\x jp} -1\right|^2 \tilde{u_p}^2,
	\end{align}
	and
	\begin{align}\label{eqn:l2-I1-3}
				\sup_{i\leq p} \left| u_i\sum_{j\leq p,j\neq i}\frac{x_i-x_j}{\x ij^3}u_j \right|
		&\geq	\sup_{i\leq p-1} \left| u_i\sum_{j\leq p-1,j\neq i}\frac{x_i-x_j}{\x ij^3}u_j + u_i\frac{x_i-x_p}{\x ip^3}u_p\right| \notag\\
		&\geq	\sup_{i\leq p-1} \left| u_i\sum_{j\leq p-1,j\neq i}\frac{x_i-x_j}{\x ij^3}u_j	\right| - \sup_{i\leq p-1} \frac{\x 1p}{\x ip^2}|\tilde{u_p}u_i| \notag\\
		&\geq	\sup_{i\leq p-1} \left| u_i\sum_{j\leq p-1,j\neq i}\frac{x_i-x_j}{\x ij^3}u_j	\right| - \sum_{j\leq p-1} \frac{\x 1p}{\x ip^2}|\tilde{u_p}u_i|.
	\end{align}
	Thus, by \eqref{eqn:l2-I1-1}, \eqref{eqn:l2-I1-2} and \eqref{eqn:l2-I1-3},
	\begin{align}\label{eqn:l2-I1-4}
		&		\sum_{i\leq p} \left| \sum_{j\leq p,j\neq i} \frac{u_j}{|x_i-x_j|} \right|^2
				+2\sup_{i\leq p} \left| u_i \sum_{j\neq i} \frac{x_i-x_j}{|x_i-x_j|^3}u_j \right| \nonumber\\
		&\geq	\frac12 \sum_{i\leq p-1} \left| \sum_{j\leq p-1,j\neq i} \frac{u_j}{\x ij}
				+\tilde{u_p}\right|^2
				+2\sup_{i\leq p-1} \left| u_i\sum_{j\leq p-1,j\neq i}\frac{x_i-x_j}{\x ij ^3}	u_j\right| \nonumber\\
		&		-2\sum_{j\leq p-1} \left| \frac{\x 1p}{\x jp} -1\right|^2 \tilde{u_p}^2
				-2\sum_{j\leq p-1} \frac{\x 1p}{\x jp^2}|\tilde{u_p}u_j| \nonumber\\
		&\geq	c_1(p,m) \( \sum_{i,j\leq p-1,j\neq i} \frac{u_j^2}{\x ij^2}+\tilde{u_p}^2\) \nonumber\\
		&		-\( 2\sum_{j\leq p-1} \left| \frac{\x 1p}{\x jp} +1\right|^2 \tilde{u_p}^2
				+2 \sum_{j\leq p-1} \frac{\x 1p}{\x jp^2}|\tilde{u_p}u_j| \).
	\end{align}
	
	Let $\beta'>\beta+1>0$ be such that
	\begin{align*}
		\frac{\beta^2}{\beta'-\beta} + \frac{\beta^2}{(1-\frac\beta{\beta'})(\beta'-\beta)}
		<	\frac{c_1(p,m)}{4p},\text{\quad and\quad} %\frac{\beta}{\beta'-\beta}<\frac13.	
		1-\frac{\beta}{\beta'}>\frac1{\sqrt2}\approx 0.707.
	\end{align*}
	
	Note that, the minus term in \eqref{eqn:l2-I1-4} can be written as
	\begin{align}\label{eqn:l2-1}
		&	2\sum_{j\leq p-1} \left| \frac{\x 1p}{\x jp} -1\right|^2 \tilde{u_p}^2 + 2\sum_{j\leq p-1} \frac{\x 1p}{\x jp^2}|\tilde{u_p}u_j| \nonumber\\
		&=	2\sum_{2\leq j\leq p-1} \( \left| \frac{\x 1p}{\x jp} -1\right|^2 \tilde{u_p}^2 + \frac{\x 1p}{\x jp^2}|\tilde{u_p}u_j|\) +2\frac{|\tilde{u_p}u_1|}{\x 1p}.
	\end{align}
	
	Considering the case $|x_1-x_p|>\beta'$, we have
	\begin{align*}
		& \left| \frac{\x 1p}{\x jp}-1\right|^2 \tilde{u_p}^2	+\frac{\x 1p}{\x jp^2} |\tilde{u_p}u_j| \\
		& \leq \frac{\(\x 1p-\x jp\)^2}{\x jp^2} \tilde{u_p}^2 + \frac{\x 1p}{\x jp^2} (\tilde{u_p}^2+u_j^2) \\
		& \leq \frac{\x 1j^2}{\x jp^2}\tilde{u_p}^2 + \frac{\x 1p}{\x jp^2}\tilde{u_p}^2 + \frac{\x 1j^2 \x 1p}{\x jp^2}\frac{u_j^2}{\x 1j^2} \\
		& \leq \frac{\beta^2}{(\beta'-\beta)^2}\tilde{u_p}^2 +\frac{\x 1j^2 \x1p}{\x pj^2} \(\tilde{u_p}^2+\frac{u_j^2}{\x 1j^2}\) \\
		& \leq \frac{\beta^2}{(\beta'-\beta)^2}\tilde{u_p}^2 + \frac{\x 1j^2}{\( 1-\frac{\x 1j}{\x 1p}\)\(\x 1p-\x 1j\)} \(\tilde{u_p}^2+\frac{u_j^2}{\x 1j^2}\) \\
		& \leq \frac{\beta^2}{(\beta'-\beta)^2}\tilde{u_p}^2 + \frac{\beta^2}{\(1-\frac{\beta}{\beta'}\)\(\beta'-\beta\)} \(\tilde{u_p}^2+\frac{u_j^2}{\x 1j^2}\) \\
		& \leq \( \frac{\beta^2}{(\beta'-\beta)^2} + \frac{\beta^2}{\(1-\frac{\beta}{\beta'}\)(\beta'-\beta)}\) \( \tilde{u_p}^2 + \frac{u_j^2}{\x 1j^2} \),
	\end{align*}
	and
	\begin{align*}
	\frac{|\tilde{u_p}u_1|}{\x 1p} \leq \frac{1}{\beta'}\(\tilde{u_p}^2+u_1^2\) =\frac1{\beta'}\(\tilde{u_p}^2+\frac{u_1^2}{|x_2-x_1|} \).
	\end{align*}
	Thus,
	\begin{align}
				\eqref{eqn:l2-1}
		&\leq	2\( \frac{\beta^2}{(\beta'-\beta)^2} + \frac{\beta^2}{\(1-\frac{\beta}{\beta'} \) 	(\beta'-\beta)} \)
				\sum_{2\leq j\leq p-1} \(\tilde{u_p}^2+\frac{u_j^2}{\x 1j^2}\)\notag\\
		&\phantom{\leq}	
		+	\frac2{\beta'} \(\tilde{u_p}^2+\frac{u_1^2}{\x 12^2} \) \notag\\
		&\leq	2\( \frac{\beta^2}{\beta'-\beta} +	\frac{\beta^2}{\(1-\frac{\beta}{\beta'} \) (\beta'-\beta)} \)
				\( (p-2)\tilde{u_p}^2 +\sum_{2\leq j\leq p-1} \frac{u_j^2}{\x 1j^2}\) \notag\\
		&\phantom{\leq}
		+	\frac{2\beta^2}{\beta'-\beta}\(\tilde{u_p}^2+\frac{u_1^2}{\x12^2}\) \notag\\
		&\leq	2\( \frac{\beta^2}{\beta'-\beta} + \frac{\beta^2}{\(1-\frac{\beta}{\beta'} \) (\beta'-\beta)} \) \times\notag\\
		&\phantom{=}\( (p-1)\tilde{u_p}^2 +\sum_{2\leq j\leq p-1} \frac{u_j^2}{\x 1j^2}+\frac{u_1^2}{\x12^2}\)  \notag\\
		&\leq	\frac{c_1(p,m)}{2p}	
				\(	(p-1)\tilde{u_p}^2+\sum_{i,j\leq p-1,i\neq j} \frac{u_j^2}{\x ij^2} \) \notag\\
		&\leq \frac{c_1(p,m)}{2} \(\tilde{u_p}^2 +\sum_{i,j\leq p-1,i\neq j} \frac{u_j^2}{\x ij^2}\)	\label{eqn:l2-I1-5}.
	\end{align}
	By \eqref{eqn:l2-I1-4} and \eqref{eqn:l2-I1-5},
	\begin{align}\label{eqn:l2-I1-6}
		&		\sum_{i\leq p} \left|\sum_{j\leq p,j\neq i}\frac{u_j}{\x ij} \right|^2 + 2\sup_{i\leq p} \left|u_i\sum_{j\leq p,j\neq i} \frac{x_i-x_j}{\x ij^3}u_j \right| \notag\\
		%&\geq	c_1(p,m) \(\sum_{i,j\leq p-1,i\neq j} \frac{u_j^2}{\x ij^2} +\tilde{u_p}^2\) - \frac{c_1(p,m)}{2}\(\sum_{i,j\leq p-1,j\neq i}\frac{u_j^2}{\x ij^2}+\tilde{u_p}^2\) \notag\\
		&\geq	\frac{c_1(p,m)}{2}\(\sum_{i,j\leq p-1,j\neq i}\frac{u_j^2}{\x ij^2}+\tilde{u_p}^2\).
	\end{align}

	Note that for $1\leq j\leq p-1$,
	\begin{align*}
				\frac{u_j^2}{\x pj^2}
		\leq	\frac{u_j^2}{\(\x p1-\x j1\)^2}
		\leq	\frac{u_j^2}{(\beta'-\beta)^2}
		\leq	\frac{u_j^2}{9\beta^2}
		\leq	\frac{u_j^2}{9\x 1j^2}.
	\end{align*}
	That is,
	\begin{align}\label{eqn:l2-2}
		\sum_{j\leq p-1} \frac{u_j^2}{\x pj^2} \leq \sum_{j\leq p-1} \frac{u_j^2}{9\x 1j^2}.
	\end{align}
	For $1\leq i\leq p-1$,
	\begin{align*}
				\frac{\x 1p^2}{\x ip^2}
		\leq	\frac{\x1p^2}{\(\x1p-\x1i\)^2}
		=		\frac1{\(1-\frac{\x 1i}{\x 1p}\)^2}
		\leq	\frac1{\(1-\frac{\beta}{\beta'}\)^2}
%		=		\(1+\frac{\beta}{\beta'-\beta}\)^2	\\
%		&\leq	(1+\frac13)^2
		<2.
	\end{align*}
	Thus,
	\begin{align*}
		\frac1{\x ip^2}	\leq	\frac2{\x 1p^2},	\quad	1\leq i\leq p-1.
	\end{align*}
	Sum over $i$ from 1 to $p-1$:
	\begin{align*}
	\sum_{i\leq p-1}\frac{1}{\x ip^2}	\leq	\frac{2p-2}{\x 1p^2}
	\end{align*}
	Thus,
	\begin{align}\label{eqn:l2-3}
		\sum_{i\leq p-1} \frac{u_p^2}{\x ip^2}	\leq	2p\frac{u_p^2}{\x 1p^2}	=	2p\tilde{u_p}^2.
	\end{align}
	Thus, by \eqref{eqn:l2-2} and \eqref{eqn:l2-3}
	\begin{align}\label{eqn:l2-I2}
				\sum_{i,j\leq p,i\neq j} \frac{u_j^2}{\x ij^2}
		&=		\sum_{i,j\leq p-1,i\neq j} \frac{u_j^2}{\x ij^2}	+\sum_{j\leq p-1,i=p} \frac{u_j^2}{\x pj^2}	+\sum_{i\leq p-1,j=p}\frac{u_p^2}{\x ip^2} \notag\\
		&\leq	\sum_{i,j\leq p-1,i\neq j} \frac{u_j^2}{\x ij^2} +\frac19 \sum_{j\leq p-1} \frac{u_j^2}{\x 1j^2} + 2p\tilde{u_p}^2 \notag\\
%		&\leq	2\sum_{i,j\leq p-1,i\neq j}\frac{u_j^2}{\x ij^2} +2p\tilde{u_p}^2 \\
		&\leq	2p\( \sum_{i,j\leq p-1,i\neq j}\frac{u_j^2}{\x ij^2} +\tilde{u_p}^2\).
	\end{align}
	
	Therefore, by \eqref{eqn:l2-I1-6} and \eqref{eqn:l2-I2}, we have
	\begin{align*}
		&		\sum_{i\leq p} \left|\sum_{j\leq p,j\neq i}\frac{u_j}{\x ij}	\right|^2 + 2\sup_{i\leq p} \left| u_i\sum_{j\leq p,j\neq i} \frac{x_i-x_j}{\x ij^3}u_j\right| \\
		&\geq	\frac{c_1(p,m)}{2} \(\sum_{i,j\leq p-1,i\neq j} \frac{u_j^2}{\x ij^2} + \tilde{u_p}^2	\) \\
		&\geq	\frac{c_1(p,m)}{4p} \sum_{i,j\leq p,i\neq j} \frac{u_j^2}{\x ij^2}.
	\end{align*}
	Thus, Conjecture~\ref{conj} holds for $\x1p>\beta'$.
	
%-------------------------------------------------------------------------------
%	\paragraph{有界}
	Now, consider the case $\x1p\leq\beta'$. we have
	$$1\leq|x_i-x_j|\leq2\beta',\quad 1\leq i,j\leq p.$$
	Let $c(p,m,x_1,\dots,x_p)$ be the infimum of
	%根??上??的论?, 我们??以??设$$ 1\leq|x_i-x_j|\leq\beta' $$ 对 $1\leq i,j\leq p$. 那么??需??明
	\begin{align*}
		\left(
		\sum_{\substack{i,j\\i\neq j}} \frac{u_j^2}{|x_i-x_j|^2}\right)^{-1}
		\times\left(
		\sum_{i\leq p} \left|\sum_{j\leq p,j\neq i,} \frac{u_j}{|x_i-x_j|}\right|^2
		+ 2\sup_{i\leq p}  \left|u_i \sum_{j\leq p,j\neq i} \frac{x_i-x_j}{|x_i-x_j|^3}u_j\right|
		\right)
	\end{align*}
	where $u_1^2+\dots u_p^2=1$.
	By the continuity of $c(p,m,x_1,\dots,x_p)$ relative to $x_1,\dots,x_p$, it suffices to prove that $c(p,m,x_1,\dots,x_p)>0$ for any $x_1,\dots,x_p$.
	
	Argue by contradiction. Assume that $c(p,m,x_1,\dots,x_p)=0$ for some given $x_1,\dots,x_p$, then by the continuity relative to $u_1,\dots,u_p$, there exist real numbers $u_1,\dots,u_p$ satisfying $u_1^2+\dots+u_p^2=1$ such that
	%???. 如果该?等???立, 那么存在满足$u_1^2+\dots+u_p^2=1$的实数$u_1,\dots,u_p$, 使得
	%By the above argument, we may assume that $$ 1\leq|x_i-x_j|\leq\beta' $$ for $1\leq i,j\leq p$. Thus it suffices to prove that $$ c(p,m,x_1,\dots,x_p)>0. $$
	%If the above inequality is not ture, then there are real numbers $u_1,\dots,u_p$ with $u_1^2+\dots+u_p^2=1$ such that
	\begin{align}
	\sum_{j\neq i} \frac{u_j}{|x_i-x_j|}=0,	\quad	1\leq i\leq p	\label{eqn:l2-iden1}\\
	u_i\sum_{j\neq i} \frac{x_i-x_j}{|x_i-x_j|^3}u_j=0,	\quad	1\leq i\leq p	\label{eqn:l2-iden2}
	\end{align}
	Since the inequality is invariant under translation of $x_1,\dots,x_p$, we may assume that $x_p=0$.
%	由于该?等?在关于$x_1,\dots,x_p$的平移下???, 我们?妨??设$x_p=0$.	
	Let
	\begin{alignat*}{3}
		t_i&=\frac{x_i}{|x_i|^2},	&\quad&	v_i&&=\frac{u_i}{|x_i|},	\quad	1\leq i\leq p-1,	\\
		t_p&=\infty,	&&	v_p&&=u_p.
	\end{alignat*}
	Note that
	\begin{align*}
		|t_i-t_j|=\frac{|x_i-x_j|}{|x_i||x_j|}, 1\leq i,j\leq p-1, i\neq j.	
	\end{align*}
	
	When $1\leq i\leq p-1$, we can rewrite \eqref{eqn:l2-iden1} as
	\begin{align}\label{eqn:l2-iden1'}
		0
		=	|x_i|\cdot\sum_{j\leq p,j\neq i}	\frac{u_j}{|x_i-x_j|}
		=	\sum_{j\leq p-1,j\neq i}	\frac{v_j}{|t_i-t_j|}+v_p,	\quad	1\leq i\leq p-1,
		%&=	\sum_{j\leq p-1,j\neq i} \frac{|t_i|^2-2\<t_i,t_j\>+|t_j|^2}{|t_i-t_j|^3} v_j +v_p,	\quad	1\leq i\leq p-1,
	\end{align}
	that is,
	\begin{align}\label{eqn:l2-iden1''}
		\sum_{j\leq p-1,j\neq i} \frac{|t_i|^2-2\<t_i,t_j\>+|t_j|^2}{|t_i-t_j|^3} v_j +v_p = 0,	\quad	1\leq i\leq p-1.
	\end{align}
	Similarly, $1\leq i\leq p-1$, we can rewrite \eqref{eqn:l2-iden2} as
	\begin{align*}
		0&=	u_i\sum_{j\neq i} \frac{x_i-x_j}{|x_i-x_j|^3}u_j \\
		&=	u_i\sum_{j\neq i,j\leq p-1} \frac{x_i-x_j}{|x_i-x_j|^3}u_j +u_i\frac{x_i}{|x_i|^3} u_p\\
		&=	\frac{v_i}{|t_i|} \sum_{j\neq i,j\leq p-1} \frac{|t_i|^3|t_j|^3}{|t_i-t_j|^3} \(\frac{t_i}{|t_i|^2}-\frac{t_j}{|t_j|^2}\) \frac{v_j}{|t_j|}
		+v_i t_i v_p,	\quad 	1\leq i\leq p-1,
	\end{align*}
	that is,
	\begin{align}\label{eqn:l2-iden3}
		v_i \sum_{j\neq i,j\leq p-1} \frac{|t_j|^2 t_i -|t_i|^2 t_j}{|t_i-t_j|^3} v_j
		+v_i t_i v_p =0,	\quad 	1\leq i\leq p-1.
	%=-u_i\frac{x_i}{|x_i|^3} u_p
	\end{align}
	Taking inner product with $t_i$, we have
	%对上???与$t_i$的内积，并化简,
	\begin{align}\label{eqn:l2-iden4}
	v_i \sum_{j\neq i,j\leq p-1} \frac{|t_j|^2 - \<t_i,t_j\>}{|t_i-t_j|^3} v_j
	+v_i v_p=0,	\quad 	1\leq i\leq p-1.
	%		=-\frac{u_i u_p}{|x_i|^3}
	\end{align}
	Taking the calculation $v_i\times\eqref{eqn:l2-iden1''} -2\times\eqref{eqn:l2-iden4}$, we get:	
	\begin{align}\label{eqn:l2-iden5}
	v_i \sum_{j\leq p-1,j\neq i} \frac{|t_i|^2-|t_j|^2}{|t_i-t_j|^3}v_j=v_iv_p,	\quad	1\leq i\leq p-1
	\end{align}
	Taking the calculation $\dfrac{1}{|t_i|^2}(t_i\times\eqref{eqn:l2-iden5}+\eqref{eqn:l2-iden3})$, we get:
	\begin{align}\label{eqn:l2-iden6}
	v_i \sum_{j\leq p-1,j\neq i} \frac{t_i-t_j}{|t_i-t_j|^3} v_j=0,	\quad	1\leq i\leq p-1.
	\end{align}
	
	By \eqref{eqn:l2-iden1'} and \eqref{eqn:l2-iden6},
	and by the assumption that $c_1(p,m,x_1,\dots,x_{p-1})>0$, we have $v_i=0,\forall i=1,\dots,p$.
	Thus $u_i=0,\forall i=1,\dots,p$, which is contrary to the assumption that $u_1^2+\dots+u_p^2=1$.
\end{proof}

%------------------------------------------------------------------------------

Before we apply Lemma~\ref{lemma2}, we need to define a useful transform.
\begin{defn}
	The Kelvin transform about some point $N\in\mathbb R^m$ is defined as follows:
	$$\mathcal K_N:	x\mapsto N+\frac{x-N}{\|x-N\|^2},$$
	where $x\in\mathbb R^m$.
\end{defn}

%-------------------------------------------------------------------------------

\begin{prop}\label{prop:invariant}
	Equations \eqref{eqn:equiv1} and \eqref{eqn:equiv2} about $u_1,\dots,u_p$
	\begin{align}
		\sum_{j\neq i} \frac{u_j}{\x{i}{j}} &= 0,	\quad 1\leq i\leq p	\tag{\ref{eqn:equiv1}}	\\%\label{eqn:equiv1}\\
		u_i\sum_{j\neq i} \frac{x_i-x_j}{\x{i}{j}^3} u_j &=0,	\quad 1\leq i\leq p	\tag{\ref{eqn:equiv2}}%\label{eqn:equiv2}
	\end{align}
	are invariant under Kelvin transform.
\end{prop}

%-------------------------------------------------------------------------------

\begin{proof}
	Fix $N\neq x_i$, $\forall 1\leq i\leq p$.
	Let
	\begin{align*}
		y_i = \mathcal K_N(x_i)= N+\frac{x_i-N}{|x_i-N|^2},\quad y_i' = y_i-N = \frac{x_i-N}{|x_i-N|^2},\quad 1\leq i\leq p.
	\end{align*}
	Note that $|y_i'| = \frac1{|x_i-N|}$ and that
	\begin{align*}
			\yy ij^2
		&=	\<y_i'-y_j',y_i'-y_j'\>	\\
		&=	\frac{1}{|x_i-N|^2|x_j-N|^2} \( |x_i-N|^2 -2\<x_i-N,x_j-N\> +|x_j-N|^2	\)	\\
		&=	\(	\frac{|x_i-x_j|}{|x_i-N||x_j-N|}	\)^2,
	\end{align*}
	thus,
	\begin{align*}
		\x ij = \frac{\yy ij}{|y_i'||y_j'|}.
	\end{align*}
	\eqref{eqn:equiv1} and \eqref{eqn:equiv2} will be changed to:
	\begin{align}
		|y_i'| \sum_{j\neq i}\frac1{\yy ij} |y_j'| u_j &= 0,	\quad 1\leq i\leq p	\label{eqn:equiv3}\\
		|y_i'|u_i \sum_{j\neq i}\frac{y_i'|y_j'|^2-y_j'|y_i'|^2}{\yy ij^3} |y_j'|u_j &= 0,	\quad 1\leq i\leq p	\label{eqn:equiv4}		
	\end{align}
%	{\color{red}
%	$u_i\times\eqref{eqn:equiv3}-2\<\eqref{eqn:equiv4},\frac{y_i'}{|y_i'|^2}\>$:
%	\begin{align}\label{eqn:p1-1}
%		|y_i'|u_i \sum_{j\neq i} \frac{|y_i'|^2-|y_j'|^2}{\yy ij^3} |y_j'| u_j = 0.
%	\end{align}
%	$\eqref{eqn:equiv4}+y_i'\times\eqref{eqn:p1-1}$:
%	\begin{align}\label{eqn:p1-2}
%		|y_i'|^3 u_i \sum_{j\neq i} \frac{y_i'-y_j'}{\yy ij^3} |y_j'|u_j
%		=	|y_i'| u_i \sum_{j\neq i} \frac{ \(y_i'|y_j'|^2-y_j'|y_i'|^2\) + y_i'\(|y_i'|^2-|y_j'|^2\)}{\yy ij^3} |y_j'|u_j =0
%	\end{align}
%两步缩?一步，改?：}
	Taking the calculation $u_iy_i'\times\eqref{eqn:equiv3}	+ \eqref{eqn:equiv4}	- 2\frac{y_i'}{|y_i'|^2}\times\<\eqref{eqn:equiv4},y_i'\>$, we get:
	\begin{align}\label{eqn:p1-2}
		|y_i'|^3 u_i \sum_{j\neq i} \frac{y_i'-y_j'}{\yy ij^3} |y_j'|u_j =0
	\end{align}
	
	Let $v_i = |y_i'|u_i$ for $i=1,\dots,p$, then \eqref{eqn:equiv3} and \eqref{eqn:p1-2} give
	\begin{align}
		\sum_{j\neq i} \frac1{\y ij}v_j &= 0,	\quad	1\leq i\leq p	\label{eqn:equiv1'}\\
		v_i \sum_{j\neq i} \frac{y_i-y_j}{\y ij^3} v_j &= 0,	\quad	1\leq i\leq p.	\label{eqn:equiv2'}
	\end{align}
	Therefore, the equations \eqref{eqn:equiv1} and \eqref{eqn:equiv2} are invariant under Kelvin transform.
\end{proof}

	Sometimes we need to send some point to infinity, e.g., $N=x_p$.
	Again, let $y_i=\mathcal K_{x_p}(x_i)$ and $y_i'=y_i-x_p$, namely,
	\begin{align*}
		y_i = x_p+\frac{x_i-x_p}{\x ip^2},\quad y_i'=\frac{x_i-x_p}{\x ip^2},\quad 1\leq i\leq p-1,
	\end{align*}
	and $y_p=\infty$.
	Then \eqref{eqn:equiv1} will be changed to:
	\begin{align}
			|y_i'| \( \sum_{1\leq j\leq p-1,j\neq i} \frac{1}{\yy ij} |y_j'|u_j +u_p\) &= 0,	\quad	1\leq i\leq p-1	\label{eqn:equiv5}\\
			\sum_{j=1}^{p-1} |y_j'|u_j &= 0,	\quad i=p.\label{eqn:equiv6}
	\end{align}
	And \eqref{eqn:equiv2} will be changed to:
	\begin{align}
		|y_i'|u_i \( \sum_{1\leq j\leq p-1,j\neq i} \frac{y_i'|y_j'|^2-y_j'|y_i'|^2}{\yy ij^3} |y_j'|u_j + y_i'u_p	\) &= 0,	\quad	1\leq i\leq p-1	\label{eqn:equiv7}\\
		u_p \sum_{j=1}^{p-1} y_j'|y_j'|u_j &= 0,	\quad	i=p\label{eqn:equiv8}.
	\end{align}
%	{\color{red}	\eqref{eqn:equiv8} can be derived from \eqref{eqn:equiv7} by summing over $i$ from 1 to $p-1$.}
%
	By a similar argument as above, %$u_i\times\eqref{eqn:equiv6}-2\<\eqref{eqn:equiv7},\frac{y_i'}{|y_i'|^2}\>$:
	taking the calculation $u_iy_i'\times\eqref{eqn:equiv5}	+ \eqref{eqn:equiv7}	- 2\frac{y_i'}{|y_i'|^2}\times\<\eqref{eqn:equiv7},y_i'\>$ gives
	\begin{align}\label{eqn:p1-3}
		|y_i'|^3 u_i \sum_{1\leq j\leq p-1,j\neq i} \frac{y_i'-y_j'}{\yy ij^3}|y_j'|u_j =0,	\quad	1\leq i\leq p-1.
	\end{align}
	Let $v_i = |y_i'|u_i$ for $1\leq i\leq p-1$, and $v_p =u_p$, then \eqref{eqn:equiv5} and \eqref{eqn:p1-3} give:
	\begin{align*}
		\sum_{1\leq j\leq p-1,j\neq i}\frac1{\y ij}v_j +v_p &= 0,	\quad 1\leq i\leq p-1	\\
		v_i \sum_{1\leq j\leq p-1,j\neq i} \frac{y_i-y_j}{\y ij^3}v_j &=0,	\quad 1\leq i\leq p-1.
	\end{align*}

Since Kelvin transform is invertible, we have the following corollary.
\begin{cor}\label{cor1}
	Given any $x_1,\dots,x_p$, the equations about $u_1,\dots,u_p$
	\begin{align*}
		\sum_{j\neq i} \frac{u_j}{\x{i}{j}} = 0,	\quad
		u_i\sum_{j\neq i} \frac{x_i-x_j}{\x{i}{j}^3} u_j =0,	\quad 1\leq i\leq p
	\end{align*}	
	have only zero solution if and only if given any $y_1,\dots,y_{p-1}$, the equations about $v_1,\dots,v_p$
	\begin{align*}
		\sum_{1\leq j\leq p-1,j\neq i}\frac1{\y ij}v_j +v_p = 0,	\quad
		v_i \sum_{1\leq j\leq p-1,j\neq i} \frac{y_i-y_j}{\y ij^3}v_j =0,	\quad 1\leq i\leq p-1.
	\end{align*}
	have only zero solution.\qed
\end{cor}

%-------------------------------------------------------------------------------

\begin{thm}\label{thm:equiv}
Let $p_0\geq4$.
If Conjecture~\ref{conj} is valid for $p\leq p_0-1$, then when $p=p_0$, the conjecture is valid if and only if for any $x_1,\dots,x_p$, equations \eqref{eqn:equiv1} and \eqref{eqn:equiv2}
\begin{align*}
	\sum_{j\neq i} \frac{u_j}{\x{i}{j}} &= 0,	\quad 1\leq i\leq p	\tag{\ref{eqn:equiv1}}\\
	u_i\sum_{j\neq i} \frac{x_i-x_j}{\x{i}{j}^3} u_j &=0,	\quad 1\leq i\leq p	\tag{\ref{eqn:equiv2}}
\end{align*}
about $u_1,\dots,u_p$ have NO non-zero solution.
\end{thm}
\begin{proof}	
	The \textit{only if} direction is obvious due to the positivity of $I_2$. Otherwise,
	assume for some $x_{10},\dots,x_{p0}$, the equations \eqref{eqn:equiv1} and \eqref{eqn:equiv2} about $u_1,\dots,u_p$ have a non-zero solution $(u_{10},\dots,u_{p0})$.
	%	Then,
	%	\begin{align*}
	%		c(p,m,x_{10},\dots,x_{p0})
	%		\leq
	%		\left(
	%		\sum_{\substack{i,j\\i\neq j}} \frac{u_j^2}{|x_i-x_j|^2}\right)^{-1}
	%		\times\left(
	%		\sum_{i\leq p} \left|\sum_{j\leq p,j\neq i,} \frac{u_j}{|x_i-x_j|}\right|^2
	%		+ 2\sup_{i\leq p}  \left|u_i \sum_{j\leq p,j\neq i} \frac{x_i-x_j}{|x_i-x_j|^3}u_j\right|
	%		\right)
	%	\end{align*}
	Substitute $x_{i0}$ and $u_{i0}$, $1\leq i\leq p$ in $I_1,I_2$:
	\begin{align*}
		I_1&=	\sum_{1\leq i\leq p} \left| \sum_{j\neq i} \frac{u_{j0}}{\x {i0}{j0}} \right|^2
		+	2\sup_{1\leq i\leq p} \left|\sum_{j\neq i} u_{i0} \frac{x_{i0}-x_{j0}}{\x {i0}{j0}^3} u_j \right|	
		=	0,	\\
		I_2&=	\sum_{\substack{i,j\\j\neq i}} \frac{u_{j0}^2}{\x{i0}{j0}^2}>0,
	\end{align*}
	which is a contradiction.
%-------------------------------------------------------------------------------

	Conversely, assume that Conjecture~\ref{conj} is NOT valid, by Lemma~\ref{lemma2} there exist $y_{10},\dots,y_{p-1,0}\in\mathbb R^m$, such that
	\begin{align*}
		c_1(p,m,y_{10},\dots,y_{p-1,0}) = 0.
	\end{align*}
	By the definition of $c_1(p,m,y_{10},\dots,y_{p-1,0})$ and the property of continuous function on a compact region, there exist $v_{10},\dots,v_{p0}\in\mathbb R$ satisfying $\sum_{i=1}^p v_{i0}^2=1$, such that
	\begin{align*}
		\sum_{1\leq j\leq p-1,j\neq i}\frac1{\y{i0}{j0}}v_{j0} +v_{p0} = 0,	\quad
		v_{i0} \sum_{1\leq j\leq p-1,j\neq i} \frac{y_{i0}-y_{j0}}{\y{i0}{j0}^3}v_{j0} =0,	\quad 1\leq i\leq p-1.
	\end{align*}
	By Corollary~\ref{cor1}, there exist $x_{10},\dots,x_{p0}\in\mathbb R^m$, such that the equations \eqref{eqn:equiv1} and \eqref{eqn:equiv2} have non-zero solutions.
\end{proof}
Now, we may prove Theorem~\ref{thm:A}.
\begin{proof}[Proof of Theorem~\ref{thm:A}]\label{proof:thmA}
	We argue by induction on $p_0$.
	
	When $p\leq3$, Conjecture~\ref{conj} holds due to Xu \cite{Xu05}, so we may apply Theorem~\ref{thm:equiv} for $p_0=4$.
	Then we see that Conjecture~\ref{conj} holds for $p=4$ if and only if \eqref{eqn:equiv1} and \eqref{eqn:equiv2} have no non-zero solution for $p=4$.
	This proves Theorem~\ref{thm:A} for $p_0=4$.
	
	Now assume the equivalence of Theorem~\ref{thm:A} holds for $4\leq p\leq p_0-1$.
	We want to show that
	Conjecture~\ref{conj} holds for $4\leq p\leq p_0$ is equivalent to that \eqref{eqn:equiv1} and \eqref{eqn:equiv2} have no non-zero solution for $4\leq p\leq p_0$.
	
	The necessity is obvious due to the positivity of $I_2$. It is left to consider the sufficiency.
	By the given condition for $4\leq p\leq p_0-1$ and the induction hypothesis, Conjecture~\ref{conj} holds for $4\leq p\leq p_0-1$.
	Applying Theorem~\ref{thm:equiv}, by the given condition for $p=p_0$, we see that Conjecture~\ref{conj} holds for $p=p_0$, hence for $p\leq p_0$.
\end{proof}

Note that equations \eqref{eqn:equiv1} and \eqref{eqn:equiv2} can be written in matrix form, i.e.,
\begin{align}\label{eqn:equiv-mat}
AU=0, \quad U\tran\dpd{A}{x_i}U=0, \quad i=1,\dots,p.
\end{align}

Let $1\leq\alpha\leq p$ be a natural number, and denote by $U_\alpha$ the column vector obtained by crossing out the $\alpha$-th row of $U$, and $A_\alpha$ the $(p-1)\times(p-1)$ matrix obtained by crossing out the $\alpha$-th row and $\alpha$-th column of $A$.

To prove Conjecture~\ref{conj} inductively by contradiction, one needs only to verify whether there is a non-zero solution  $U=(u_1,\dots,u_p)\tran$  of equation \eqref{eqn:equiv-mat} with $u_i\neq0$ for all $i=1,\cdots,p$.
Since if $u_\alpha=0$ for some $1\leq\alpha\leq p$, then \eqref{eqn:equiv-mat} yields
\begin{align*}
A_\alpha U_\alpha=0, \quad U_\alpha\tran\dpd{A_\alpha}{x_i}U_\alpha=0, \quad 1\leq i\neq\alpha\leq p,
\end{align*}
which has the same form as \eqref{eqn:equiv-mat}.
By the inductive assumption that the conjecture holds for $p\leq p_0-1$ and Theorem~\ref{thm:equiv}, we have $U_\alpha=0$, hence $U=0$.
The claim follows from the contradiction.

In summary, we have the following corollary.
\begin{cor}\label{cor:nonzero}
	If the conjecture holds for $p\leq p_0-1$ and does not hold for $p=p_0$, then any component of a non-zero solution to \eqref{eqn:equiv-mat} is non-zero.
\end{cor}
%-------------------------------------------------------------------------------
%-------------------------------------------------------------------------------
%-------------------------------------------------------------------------------

\section{Further equivalent conditions and partial results}\label{sec:3}
In this section, we shall continue to use Kelvin transform to simplify the conjecture even further and prove it for the basic case $m=1$ with arbitrary $p$, and $p=4,5$ with arbitrary $m$.
Again, assuming that the conjecture holds for $p\leq p_0-1$ ($p_0\geq4$), we shall consider the case when $p=p_0$.

Note that we can always embed the ambient Euclidean space $\mathbb R^m$, where the $p$ points $x_1,\dots,x_p$ lie in, into the first $m$-components of a larger space $\mathbb R^{m+1}$.
Then we apply the Kelvin transform about the point $N=(0,\dots,0,1)\in\mathbb R^{m+1}\setminus\mathbb R^m$.
Note that, this map is, in fact, a stereographic projection.
The image of $\mathbb R^m\hookrightarrow\mathbb R^{m+1}$ is an $m$-sphere centered at $\frac12N=(0,\dots,0,\frac12)$ with radius $\frac12$ in $\mathbb R^{m+1}$.

Let $y=\mathcal K_N(x)$ denote the image of $x$. By Proposition~\ref{prop:invariant}, equations \eqref{eqn:equiv1} and \eqref{eqn:equiv2} become equations \eqref{eqn:equiv1'} and \eqref{eqn:equiv2'}:
\begin{align}
	\sum_{j\neq i} \frac1{\y ij}v_j &= 0,	\quad	1\leq i\leq p,	\tag{\ref{eqn:equiv1'}}\\
	v_i \sum_{j\neq i} \frac{y_i-y_j}{\y ij^3} v_j &= 0,	\quad	1\leq i\leq p.	\tag{\ref{eqn:equiv2'}}
\end{align}
Since these equations are invariant under translation and scaling of $y_1,\dots,y_n$, we may assume that $y_1,\dots,y_p$ lie on the unit $m$-sphere $\mathbb S^m$.
%\in S^m\hookrightarrow\mathbb R^{m+1}$, and $|y_i|=1$, $1\leq i\leq p$.

%In fact, by the argument above, equation \eqref{eqn:equiv2'} is equivalent to
%\begin{align}
%	\sum_{j\neq i} \frac{y_i-y_j}{\y ij^3} v_j = 0,	\quad	1\leq i\leq p.	\label{eqn:equiv2''}
%\end{align}

%-------------------------------------------------------------------------------

Now, we prove the following theorem.
\begin{thm}\label{thm:equiv2}
	Let $p_0\geq4$.
	If the conjecture holds for $p\leq p_0-1$, then when $p=p_0$, the conjecture is valid if and only if given any distinct $y_1,\dots,y_p\in \mathbb S^m$, the equations
	\begin{align}
		\sum_{j\neq i} \frac{y_i-y_j}{\y ij^3} v_j =0,	\quad 1\leq i\leq p	\tag{\ref{eqn:equiv2''}}
	\end{align}
	 for $v_1,\dots,v_p$ have only zero solution.
\end{thm}

\begin{proof}
	Assume that the conjecture is false when $p=p_0$, then by Theorem~\ref{thm:equiv}, equations \eqref{eqn:equiv1} and \eqref{eqn:equiv2} have a non-zero solution, say $U=(u_1,\dots,u_p)\tran$.
	By Corollary~\ref{cor:nonzero}, each component of $U$ is non-zero.
	By Proposition~\ref{prop:invariant}, equations \eqref{eqn:equiv1'} and \eqref{eqn:equiv2'} have a non-zero solution $V=(v_1,\dots,v_p)\tran$, and each component $v_i$ is also non-zero.
	Thus we can derive equation \eqref{eqn:equiv2''} from \eqref{eqn:equiv2'} by dividing $v_i$ on both sides of the $i$-th equation for $1\leq i\leq p$.
	Therefore, equation \eqref{eqn:equiv2''} have a non-zero solution.
	
	Conversely, assume that equation \eqref{eqn:equiv2''} have a non-zero solution, say $V=(v_1,\dots,v_p)\tran$, we are going to show that the conjecture is false when $p=p_0$ by showing the existence of non-zero solutions for equations \eqref{eqn:equiv1'} and \eqref{eqn:equiv2'}.
	In fact, $V$ clearly satisfies equation \eqref{eqn:equiv2'}. For \eqref{eqn:equiv1'}, %for $1\leq i\leq p$, by multiplying $v_i$ on both sides of the $i$-th equation of \eqref{eqn:equiv2''}, we have the $i$-th equation of \eqref{eqn:equiv2'}.
	note that
	\begin{align*}
		\y ij^2 = \<y_i-y_j,y_i-y_j\> = 2-2\<y_i,y_j\>.
	\end{align*}
	Taking inner product of \eqref{eqn:equiv2''} with $2y_i$, we get
	\begin{align*}
		0
		=	\sum_{j\neq i} \frac{2-2\<y_i,y_j\>}{\y ij^3} v_j
		=	\sum_{j\neq i} \frac{\y ij^2}{\y ij^3} v_j
		=	\sum_{j\neq i} \frac{1}{\y ij} v_j,	\quad 1\leq i\leq p	%\label{eqn:equiv2''}
	\end{align*}
	Thus, $V$ satisfies \eqref{eqn:equiv1'} as well. In conclusion, $V$ is also non-zero a solution for equations \eqref{eqn:equiv1'} and \eqref{eqn:equiv2'}.
\end{proof}
\begin{proof}[Proof of Theorem~\ref{thm:B}]
	The proof is similar to that of Theorem~\ref{thm:A} and thus left to the reader.
\end{proof}
%-------------------------------------------------------------------------------
%-------------------------------------------------------------------------------

\subsection{Case 1. $m=1$}
In this subsection, we prove Conjecture~\ref{conj} when $m=1$ using Theorem~\ref{thm:equiv2}.

\begin{thm}\label{thm-m1}
	Conjecture~\ref{conj} is valid for $m=1$ and arbitrary $p$.
\end{thm}
\begin{proof}
	It suffices to prove the case $p=p_0\geq4$, under the assumption that the cases $p\leq p_0-1$ hold.
	
%	To avoid confusion between the summation index $i$ and the imaginary unit $\mi$, we shall change the index $i$ to $k$.
	
	Without lost of generality, we may assume $y_k=\me^{\mi\alpha_k}\in \mathbb{S}^1\hookrightarrow\mathbb C$ for $1\leq k\leq p$ with	
	$$	0\leq\alpha_1<\alpha_2<\dots<\alpha_p<2\pi.	$$
	We shall prove, the equations
	\begin{align}\label{eqn:equiv2'''}
		\sum_{j\neq k} \frac{y_k-y_j}{\y kj^3} v_j =0,	\quad 1\leq k\leq p	
	\end{align}
	for $v_1,\dots,v_p$ have only zero solution.
	We prove this by contradiction.
	
	Assume that $V=(v_1,\dots,v_p)\tran$ is a non-zero solution for equation \eqref{eqn:equiv2'''}, then each component of $V$ is non-zero by a similar argument as in Corollary~\ref{cor:nonzero}.
	Notice the following identities:
	%\begin{align*}
	%	\y kj^2 = \<y_k-y_j,y_k-y_j\>	= 2-2\<y_k,y_j\>,
	%\end{align*}
	\begin{align*}
			y_k-y_j = \me^{\mi\alpha_k}-\me^{\mi\alpha_j}
		&=	(\cos\alpha_k-\cos\alpha_j)+\mi(\sin\alpha_k-\sin\alpha_j)	\\
		&=	-2\sin\frac{\alpha_k+\alpha_j}{2}\sin\frac{\alpha_k-\alpha_j}{2} +2\mi\cos\frac{\alpha_k+\alpha_j}{2}\sin\frac{\alpha_k-\alpha_j}{2}	\\
		&=	2\sin\frac{\alpha_k-\alpha_j}{2}	(-\sin\frac{\alpha_k+\alpha_j}{2}+\mi\cos\frac{\alpha_k+\alpha_j}{2})	\\
		&=	2\sin\frac{\alpha_k-\alpha_j}{2}	\me^{\mi\frac{\alpha_k+\alpha_j+\pi}{2}},	\\
		\y kj = |\me^{\mi\alpha_k}-\me^{\mi\alpha_j}| &= 2\sgn(k-j)\sin\frac{\alpha_k-\alpha_j}{2}>0.
	\end{align*}
	Substituting the identities above in \eqref{eqn:equiv2'''}, we obtain:
	\begin{align*}
	%	\sum_{j\neq k} \frac{2\sin\frac{\alpha_k-\alpha_j}{2}	\me^{\mi\frac{\alpha_k+\alpha_j+\pi}{2}}}{(2\sgn(k-j)\sin\frac{\alpha_k-\alpha_j}{2})^3} v_j =0,	\quad 1\leq k\leq n	 \\
		%&	\sgn(k-s) \frac{	\me^{\mi\frac{\alpha_s}{2}}}{(\sin\frac{\alpha_k-\alpha_s}{2})^2} v_s	+\sum_{j\neq k,s} \sgn(k-j) \frac{	 \me^{\mi\frac{\alpha_j}{2}}}{(\sin\frac{\alpha_k-\alpha_j}{2})^2} v_j	\\
		%&=	
		\sum_{j\neq k} \sgn(k-j) \frac{	\me^{\mi\frac{\alpha_j}{2}}}{(\sin\frac{\alpha_k-\alpha_j}{2})^2} v_j
		=	0,	\quad 1\leq k\leq p.
	\end{align*}
	
	For each $1\leq s\leq p$, multiplying $\me^{\mi\frac{-\alpha_s}{2}}$ on both sides of all equations but the $s$-th one, we have
	%对上?的除第$p$个方程以外的方程乘以$\me^{\mi\frac{-\alpha_p}{2}}$, 有
	$$
	\sgn(k-s)\frac{1}{(\sin\frac{\alpha_k-\alpha_s}{2})^2}v_s	+	\sum_{j\neq k,s}\sgn(k-j)\frac{	\me^{\mi\frac{\alpha_j-\alpha_s}{2}}}{(\sin\frac{\alpha_k-\alpha_j}{2})^2} v_j=0,	 \quad	1\leq k\neq s\leq p.
	$$
	Taking the imaginary part, we get:
	\begin{align}\label{eqn:imagine}
		\sum_{j\neq k,s} \sgn(k-j) \frac{	\sin\frac{\alpha_j-\alpha_s}{2}}{(\sin\frac{\alpha_k-\alpha_j}{2})^2} v_j=0,	\quad	1\leq k\neq s\leq p.
	\end{align}
	
	%-------------------------------------------------------------------------------
	
	Note that \eqref{eqn:equiv2'''} is also equivalent to
	\begin{align*}
		%\sgn(k-s)	\frac{\me^{\mi\alpha_k}-\me^{\mi\alpha_s}}{(\sin\frac{\alpha_k-\alpha_s}{2})^3} v_s
		\sum_{j\neq k}\sgn(k-j) \frac{\me^{\mi\alpha_k}-\me^{\mi\alpha_j}}{(\sin\frac{\alpha_k-\alpha_j}{2})^3} v_j =0,	\quad 1\leq k\leq p.	
	\end{align*}
	%$\forall s:1\leq s\leq p$,
	Multiplying $v_k\sin\frac{\alpha_k-\alpha_s}{2}\me^{-\mi\frac{\alpha_k}{2}}$ on the equation above, and setting $$w_k=\me^{-\mi\frac{\alpha_k}{2}}v_k, 	 \quad 1\leq k\leq p,	$$  we get:
	%\begin{align*}
	%%	&\sum_{j\neq k}\sgn(k-j) \frac{\me^{\mi\alpha_k}-\me^{\mi\alpha_j}}{(\sin\frac{\alpha_k-\alpha_j}{2})^3} w_jw_k\sin\frac{\alpha_k-\alpha_p}{2}\me^{-\mi\frac{\alpha_k}{2}} =0,	 \quad k\neq p,1\leq k\leq n	\\
	%	\sgn(k-p) \frac{\me^{\mi\alpha_k}-\me^{\mi\alpha_p}}{(\sin\frac{\alpha_k-\alpha_p}{2})^2} w_pw_k\me^{-\mi\frac{\alpha_k}{2}} +
	%	\sum_{j\neq k,p} \sgn(k-j) \frac{	\sin\frac{\alpha_k-\alpha_p}{2}} {(\sin\frac{\alpha_k-\alpha_j}{2})^2} 2\mi\me^{\mi\frac{\alpha_j}{2}}w_jw_k =0,	\quad k\neq p,1\leq k\leq n.
	%\end{align*}
	$$
		\sgn(k-s) \frac{\me^{\mi\alpha_k}-\me^{\mi\alpha_s}}{(\sin\frac{\alpha_k-\alpha_s}{2})^2} v_sw_k	+\sum_{j\neq k,s} \sgn(k-j) \frac{	\sin\frac{\alpha_k-\alpha_s}{2}} {(\sin\frac{\alpha_k-\alpha_j}{2})^2} 2\mi\me^{\mi\frac{\alpha_j}{2}}v_jv_k =0,	\quad 1\leq k\neq s\leq p.
	$$
	Taking sum over $k$, we get
	\begin{align*}
	0	
	&=	\sum_{k\neq s}\sgn(k-s) \frac{\me^{\mi\alpha_k}-\me^{\mi\alpha_s}}{(\sin\frac{\alpha_k-\alpha_s}{2})^2} v_sw_k +
	\sum_{k\neq s}\sum_{j\neq k,s} \sgn(k-j) \frac{	\sin\frac{\alpha_k-\alpha_s}{2}} {(\sin\frac{\alpha_k-\alpha_j}{2})^2} 2\mi\me^{\mi\frac{\alpha_j}{2}}v_jv_k	\\
	&\xlongequal{\text{Exchange the order}}
	v_s\sum_{k\neq s}\sgn(k-s) \frac{\me^{\mi\alpha_k}-\me^{\mi\alpha_s}}{(\sin\frac{\alpha_k-\alpha_s}{2})^2} w_k +
	\sum_{j\neq s}\sum_{k\neq j,s} \sgn(k-j) \frac{	\sin\frac{\alpha_k-\alpha_s}{2}} {(\sin\frac{\alpha_k-\alpha_j}{2})^2} 2\mi\me^{\mi\frac{\alpha_j}{2}}v_jv_k	\\
	&=	v_s\sum_{k\neq s}\sgn(k-s) \frac{\me^{\mi\alpha_k}-\me^{\mi\alpha_s}}{(\sin\frac{\alpha_k-\alpha_s}{2})^2} w_k +
	\sum_{j\neq s}	2\mi\me^{\mi\frac{\alpha_j}{2}}v_j \sum_{k\neq j,s} \sgn(k-j) \frac{	\sin\frac{\alpha_k-\alpha_s}{2}} {(\sin\frac{\alpha_k-\alpha_j}{2})^2} v_k	\\
	&\xlongequal{\eqref{eqn:imagine}} v_s\sum_{k\neq s}\sgn(k-s) \frac{\me^{\mi\alpha_k}-\me^{\mi\alpha_s}}{(\sin\frac{\alpha_k-\alpha_s}{2})^2} w_k.
	\end{align*}
	That is,
	$$v_s\sum_{k\neq s}\sgn(k-s) \frac{\me^{\mi\alpha_k}-\me^{\mi\alpha_s}}{(\sin\frac{\alpha_k-\alpha_s}{2})^2} w_k=0,\quad1\leq s\leq p,$$
	or equivalently,
	\begin{align*}
		\sum_{k\neq s}\sgn(k-s) \frac{y_k-y_s}{\y ks^2} w_k=0,\quad1\leq s\leq p.
	\end{align*}
	Taking inner product with $2y_s$ gives
	\begin{equation}\label{eq-CW0}
		\sum_{k\neq s}\sgn(k-s) w_k=0,\quad1\leq s\leq p.
	\end{equation}
	In other words, setting $$C=(\sgn(k-s))_{k,s} =\begin{pmatrix}
	0&1&1&\dots&1\\
	-1&0&1&\dots&1\\
	-1&-1&0&\dots&1\\
	\vdots&\vdots&\vdots&\ddots&\vdots\\
	-1&-1&-1&\dots&0
	\end{pmatrix},	\quad W=(w_1,w_2,\dots,w_p)\tran\in\mathbb C^p,$$
	we have $CW=0$ by (\ref{eq-CW0}).
	
	When $p$ is even, one easily sees that $\det C=1$ by expanding the determinant, thus $W=0$, and hence $V=0$, a contradiction.
	
	When $p$ is odd, $\det C=0$. Evidently, we can find a non-zero $(p-1)\times(p-1)$ minor, and thus $\rank C=p-1$.
	We have
	$$
		CW = C\begin{pmatrix}\me^{\mi\frac{\alpha_1}{2}v_1}\\\me^{\mi\frac{\alpha_2}{2}v_2}\\\vdots\\\me^{\mi\frac{\alpha_p}{2}v_p}
		\end{pmatrix}=0,
	$$ in this case.
	Consider the real and imaginary parts of $W$:
	$$
		C\begin{pmatrix}\cos\frac{\alpha_1}{2}v_1\\\cos\frac{\alpha_2}{2}v_2\\\vdots\\\cos\frac{\alpha_p}{2}v_p
		\end{pmatrix}=0,\quad
		C\begin{pmatrix}\sin\frac{\alpha_1}{2}v_1\\\sin\frac{\alpha_2}{2}v_2\\\vdots\\\sin\frac{\alpha_p}{2}v_p
		\end{pmatrix}=0.
	$$
	Since $V\neq 0$,
	$$
	\begin{pmatrix}\cos\frac{\alpha_1}{2}v_1\\\cos\frac{\alpha_2}{2}v_2\\\vdots\\\cos\frac{\alpha_p}{2}v_p
	\end{pmatrix},	\text{\quad and\quad}
	\begin{pmatrix}\sin\frac{\alpha_1}{2}v_1\\\sin\frac{\alpha_2}{2}v_2\\\vdots\\\sin\frac{\alpha_p}{2}v_p
	\end{pmatrix}
	$$
	are two solutions for $CZ=0$, where $Z=(z_1,\dots,z_p)$ is a real variable.
	But since
	$$
		\begin{vmatrix}
			\cos\frac{\alpha_1}{2}v_1&\cos\frac{\alpha_2}{2}v_2\\
			\sin\frac{\alpha_1}{2}v_1&\sin\frac{\alpha_2}{2}v_2
		\end{vmatrix}
		=\sin\frac{\alpha_2-\alpha_1}{2}v_1v_2\neq0,
	$$
	the two solutions are linearly independent, which is contradictory to the previous conclusion that $\rank C=p-1$.
	
	In either case, we yield a contradiction, which shows that the assumption that $V=(v_1,\dots,v_p)\tran$ is a non-zero solution for equation \eqref{eqn:equiv2'''} is false, proving the conjecture.
\end{proof}

%-------------------------------------------------------------------------------
%-------------------------------------------------------------------------------

\subsection{Case 2. $p=4,5$}
In this subsection, we deal with the cases $p=4,5$. We manage to reduce them down to $m=1$.
\begin{thm}\label{thm-p45}
	Conjecture~\ref{conj} is valid for $p=4,5$ and arbitrary $m$.
\end{thm}
\begin{proof}
Assume that $U=(u_1,\dots,u_p)\tran$ is a non-zero solution for equations \eqref{eqn:equiv1} and \eqref{eqn:equiv2} for some $x_1,\dots,x_p\in\mathbb R^m$.
It suffices to \textit{transform} $x_1,\dots,x_p$ to lie on the same line, yielding a contradiction by Theorems~\ref{thm:A} and \ref{thm-m1}.

\subsubsection{$p=4$}
We may assume that $m=3$ by Remark~\ref{rmk:m}, and that each component of $U$ is non-zero by Corollary~\ref{cor:nonzero}.
Thus, we have
\begin{align}
	\sum_{j\neq i} \frac{u_j}{\x{i}{j}} &= 0,	\quad 1\leq i\leq p	\label{eqn:equiv1-p}\\
	\sum_{j\neq i} \frac{x_i-x_j}{\x{i}{j}^3} u_j &=0,	\quad 1\leq i\leq p	\label{eqn:equiv2-p}.
\end{align}
Choose the $x$-axis of $\mathbb R^3$ as the line passing $x_1$  and $x_2$.
If $x_3$ and $x_4$ are on $x$-axis, then we have completed the proof by contradiction.

If $x_3,x_4$ are not on the $x$-axis, then from \eqref{eqn:equiv2-p}, we claim that $x_1,\dots,x_4$ must be on the same 2-dimensional plane.
Otherwise, if $x_4$ is outside the plane defined by $x_1,x_2,x_3$, then $u_4$ must be zero by equation \eqref{eqn:equiv2-p} when $i = 1$, a contradiction.

Since equations \eqref{eqn:equiv1-p} and \eqref{eqn:equiv2-p} are invariant under the Kelvin
transform by Proposition \ref{prop:invariant}, we can make $x_1 ,x_2 ,x_3$ co-line by transforming the circle through them into a line.
If $x_4$ is not on the line of $x_1 ,x_2 ,x_3$, then $u_4 = 0$ in the above equation for $i = 4$, a contradiction again. Finally we still come to the co-line case and thus complete the proof for $p=4$.
%Now all the 4 points lie on a line.
%Therefore, the conjecture holds for $p=4$, $\forall m$.

\subsubsection{$p=5$}
We may assume that $m\leq 4$.
Since the case $p=4$ has been verified, we may assume that $U\in\mathbb{R}^p$ $(u_i\neq0)$ is a solution to equations \eqref{eqn:equiv1-p} and \eqref{eqn:equiv2-p} again.

Denote $E_0:=\Span\{x_1-x_i\mid i=2,3,4,5\}$.

If $\dim E_0=4$, i.e., $x_1-x_5\notin E_1:=\Span\{x_1-x_i\mid i=2,3,4\}$, then $\dist{x_1-x_5}{E_1}:=\inf_{y\in E_1}|y-(x_1-x_5)|$ can be achieved at some point $y_0\in E_1$.
So $y_0-(x_1-x_5)\perp E_1$.
Let $i=1$ in \eqref{eqn:equiv2-p}, and make inner product between it and $y_0-(x_1-x_5)$, we have $u_5=0$, a contradiction. So $x_1-x_5\in E_1$, and thus $E_0=E_1$, $\dim E_0\leq 3$.

Now if $\dim E_0=3$, then we can assume the 5 points are in the same 3-dimensional subspace
and we can pick up 3 points such that the other 2 locate at the same side of the plane determined by the 3 points.
Suppose $x_1-x_4$ and $x_1-x_5$ are on the same side of the plane $E_2:=\Span\{x_1-x_2,x_1-x_3\}$.
We can write $$x_1-x_5=c(x_1-x_4)+a(x_1-x_2)+b(x_1-x_3)$$ for constants $a,b$ and $c>0$.
And we can find $y_1\in E_2$ such that $\alpha:=y_1-(x_1-x_4)\perp E_2$.
Let $i=1,2,3$ in \eqref{eqn:equiv2-p}, and make inner products between each of them and $\alpha$, then
\begin{align*}
	\frac{\<\alpha,x_1-x_4\>}{\x14^3} u_4	+	\frac{\<\alpha,x_1-x_5\>}{\x15^3} u_5	&=0,	\\
	\frac{\<\alpha,x_1-x_4\>}{\x24^3} u_4	+	\frac{\<\alpha,x_1-x_5\>}{\x25^3} u_5	&=0,	\\
	\frac{\<\alpha,x_1-x_4\>}{\x34^3} u_4	+	\frac{\<\alpha,x_1-x_5\>}{\x35^3} u_5	&=0.
\end{align*}
Note that $\<\alpha,x_1-x_5\> = c\<\alpha,x_1-x_4\><0$, thus we have
\begin{align*}
	\frac{\x15}{\x14}	=\frac{\x25}{\x24}	=\frac{\x35}{\x34}	:=K%:=-\sqrt[3]{\frac{cu_5}{u_4}},	
	,\quad	u_5 = -\frac1c K^3 u_4.
\end{align*}
Let $i=4,5$ in \eqref{eqn:equiv1-p}, then
\begin{align*}
	\frac{u_1}{\x41}	+\frac{u_2}{\x42}	+\frac{u_3}{\x43}	-\frac{K^3u_4}{c\x45}	&= 0,	\\
	\frac{u_1}{K\x41}	+\frac{u_2}{K\x42}	+\frac{u_3}{K\x43}	+\frac{u_4}{\x45}	&= 0.
\end{align*}
Then we have $c=-K^2<0$,
a contradiction. This shows that $x_1-x_4\in E_2$, i.e., $E_0=E_1=E_2$. So $\dim E_0\leq2$.

Now, we have all the five points on a 2-dimensional plane.
Pick up 3 of them such that the circle they form enclose the other 2 points. Under a proper Kelvin transform, we may assume that the 3 lie on a line, and the other 2 lie in the same side of the line.
By similar contradiction argument as above, we conclude that all the 5 points lie on the same line and thus $U=0$ by Theorem \ref{thm-m1}, hence a contradiction to Theorem \ref{thm:A}.
This completes the proof for $p=5$.
\end{proof}

%%%%%%%%%%%%%%%%%%%%%%%
\begin{ack}
The authors would like to thank the referees for their valuable comments.
\end{ack}
%-------------------------------------------------------------------------------
%-------------------------------------------------------------------------------
%-------------------------------------------------------------------------------

\end{document}